\documentclass{amsart}

\usepackage{amsmath,amssymb,amsthm}
\usepackage{hyperref}
\usepackage{xcolor}

\newtheorem{prop}{Proposition}[subsection]
\newtheorem{thm}[prop]{Theorem}
\newtheorem{thmintr}{Theorem}[section]
\newtheorem{lem}[prop]{Lemma}
\newtheorem{cor}[prop]{Corollary}

\theoremstyle{definition}

\newtheorem{rem}[prop]{Remark}
\newtheorem{exwith}[prop]{Example}

\newtheorem*{ack}{Acknowledgement}

\def\co{\colon\thinspace}

\newcommand{\C}{\mathbb C}

\newcommand{\rmd}{\mathrm d}
\newcommand{\D}{\mathbb D}
\newcommand{\DD}{\mathcal D}

\newcommand{\rme}{\mathrm e}

\newcommand{\EE}{\mathcal E}

\newcommand{\Hp}{\mathbb H}
\newcommand{\HH}{\mathcal H}

\newcommand{\rmi}{\mathrm i}

\newcommand{\N}{\mathbb N}

\newcommand{\R}{\mathbb R}

\newcommand{\TT}{\mathcal T}

\newcommand{\VV}{\mathcal V}

\newcommand{\WW}{\mathcal W}

\newcommand{\Z}{\mathbb Z}

\newcommand{\lra}{\longrightarrow}
\newcommand{\ra}{\rightarrow}

\DeclareMathOperator{\Area}{\mathrm{Area}}

\DeclareMathOperator{\dR}{\mathrm{dR}}
\DeclareMathOperator{\dist}{\mathrm{dist}}

\DeclareMathOperator{\grad}{\mathrm{grad}}

\DeclareMathOperator{\Hess}{\mathrm{Hess}}

\DeclareMathOperator{\hh}{\mathrm{Hofer}}

\DeclareMathOperator{\II}{\mathrm{II}}

\DeclareMathOperator{\length}{\mathrm{length}}
\DeclareMathOperator{\loc}{\mathrm{loc}}

\DeclareMathOperator{\proj}{\mathrm{proj}}

\DeclareMathOperator{\Reg}{\mathrm{Reg}}

\begin{document}

\author{Youngjin Bae}
\address{
Research Institute for Mathematical Sciences,
Kyoto University, Kyoto 606-8317, Japan
}
\email{ybae@kurims.kyoto-u.ac.jp}

\author{Kevin Wiegand}
\author{Kai Zehmisch}
\address{Kevin Wiegand and Kai Zehmisch, Justus-Liebig-Universit{\"a}t Giessen,
Mathematisches Institut, Arndtstrasse 2, D-35392 Giessen, Germany}
\email{Kevin.E.Wiegand@math.uni-giessen.de, Kai.Zehmisch@math.uni-giessen.de}

\title[Virtually contact structures]
{Periodic orbits in virtually contact structures}

\date{28.08.2018}

\begin{abstract}
  We prove that certain non-exact
  magnetic Hamiltonian systems
  on products of closed hyperbolic surfaces
  and with a potential function of large oscillation
  admit non-constant contractible periodic solutions
  of energy below the Ma\~n\'e critical value.
  For that we develop a theory of holomorphic curves
  in symplectisations of non-compact contact manifolds
  that arise as the covering space of a virtually contact structure
  whose contact form is bounded
  with all derivatives up to order three.
\end{abstract}

\thanks{YB is supported by IBS-R003-D1.
KW is supported by the SFB 878 - Groups, Geometry and Actions.
KZ is partially supported by the SFB/TRR 191 - Symplectic Structures in Geometry,
Algebra and Dynamics and DFG grant ZE 992/1-1.}

\maketitle

\tableofcontents


\section{Introduction\label{sec:intro}}

One of the important questions
in contact geometry is
whether the Reeb vector field
of a given contact form $\alpha$
on a contact manifold
$(M',\xi=\ker\alpha)$ admits a periodic solution.
The existence of closed Reeb orbits
on closed contact manifolds $(M',\xi)$
was conjectured by Weinstein \cite{wein79}.
For compact contact type hypersurfaces in $\R^{2n}$
the conjecture was verified by Viterbo \cite{vit87}
and for closed $3$-dimensional contact manifolds
by Taubes \cite{tau07}.
Hofer \cite{hof93} used finite energy holomorphic planes
in symplectisations to verify the Weinstein conjecture
in many instances of $3$-dimensional contact manifolds.
Based on Hofer's finite energy holomorphic curves
the Weinstein conjecture was shown to hold true
for a large class of higher-dimensional closed contact manifolds,
see \cite{ach05,ah09,dgz14,dgz,gz12,gz16a,gnw16,nr11,sz17}.

Non-closed contact manifolds $(M',\xi)$
in general admit aperiodic Reeb vector fields.
For example the Reeb flow of the Reeb vector field
$\partial_z$ of the contact form
$\rmd z+\mathbf{x}\rmd\mathbf{y}$
on $\R^{2n-1}$ is linear.
Imposing additional topological conditions on the manifold $M'$
and asymptotic conditions on the contact form $\alpha$
existence of periodic Reeb orbits can be shown
in many interesting cases, see \cite{bprv16,bpv09,gz16b,sz16}.

In this article we will consider open contact manifolds $(M',\xi)$
that are induced by so-called virtually contact structures.
This means that $M'$ is the total space
of a covering $\pi\co M'\ra M$
of a closed odd-dimensional symplectic manifold $(M,\omega)$
such that $\pi^*\omega=\rmd\alpha$
and $\alpha$ admits uniform lower and upper bounds
with respect to a lifted metric,
cf.\ Section \ref{subsec:virtcontstr}.
We will consider non-trivial
virtually contact structures exclusively
whose odd-dimensional symplectic form $\omega$
itself is not the exterior differential of a contact form on $M$.

It was asked by Paternain
whether the Reeb vector fields
of virtually contact structures
admit periodic orbits.
Under an additional $C^3$-bound on the contact form $\alpha$
with respect to the Levi--Civita connection of the lifted metric
we are able to make Hofer's theory about finite energy planes work
for virtually contact manifolds
and answer Paternain's question
affirmatively in the following cases:

\begin{thmintr}
\label{thm:thmaboutvcs}
 Let $(M',\xi=\ker\alpha)$ be the total space
 of a virtually contact structure
 on a closed odd-dimensional
 symplectic manifold $(M,\omega)$.
 Assume that the contact form $\alpha$
 is $C^3$-bounded.
 Then the Reeb vector field of $\alpha$ on $M'$
 admits a contractible periodic orbit
 provided that one of the following conditions
 for the $(2n-1)$-dimensional contact manifold
 $(M',\xi)$ is satisfied:
 \begin{enumerate}
 \item 
   $n=2$ and $\xi$ is overtwisted,
 \item
   $n=2$ and $\pi_2M'\neq0$,
 \item 
   $n\geq3$ and $(M',\xi)$ contains
   a Legendrian open book with boundary,
 \item 
   $n\geq3$ and $(M',\xi)$ contains
   the upper boundary of the standard
   symplectic handle of index $1\leq k\leq n-1$
   whose belt sphere $S^{2n-1-k}\subset M'$
   represents a non-trivial element in
   \begin{enumerate}
     \item
      $\pi_{2n-2}M'$ if $k=1$,
     \item
      $\pi_3M'$ if $n=3$ and $k=2$,
     \item
      $\pi_4M'$ if $n=4$ and $k=3$,
     \item
      the oriented bordism group
      $\Omega^{SO}_{2n-1-k}M'$ if $k\geq2$,
   \end{enumerate}
 \item 
   $n\geq3$ and $(M',\xi)$
   is obtained by a covering contact connected sum
   (see \cite{wz16})
   such that the underlying connected sum
   decomposition of $M$ is non-trivial
   and $\omega$ is not the exterior differential
   of a contact form on $M$.
 \end{enumerate}
\end{thmintr}

We prove Theorem \ref{thm:thmaboutvcs}
in Section \ref{subsec:germofhlomdiscs}.
Before discussing the ingredients of the proof
we mention applications to Hamiltonian dynamics:
In classical mechanics one studies the motion
of charged particles on a configuration space $Q$
under the influence of a magnetic field $\sigma$
and a conservative force field $V$.
It is assumed that a Riemannian metric is chosen on $Q$
so that a Levi--Civita connection $\nabla$,
the gradient of $V$, and the vector potential $B$
given by $\langle B\,.\,,\,.\,\rangle=\sigma$
are defined.
The equations of motion
\[
\nabla_{\dot{\mathbf{q}}}\dot{\mathbf{q}}=
-\grad V(\mathbf{q})+B(\mathbf{q})\dot{\mathbf{q}}
\]
are determined by the Hamiltonian function
\[
H(\mathbf{q},\mathbf{p})=
\frac12|\mathbf{p}|^2+V(\mathbf{q})
\]
given by the sum of kinetic
and potential energy
and the symplectic form
\[
\rmd(\mathbf{p}\,\rmd\mathbf{q})+
\sigma_{ij}(\mathbf{q})\,\rmd q^i\wedge\rmd q^j
\]
taking local coordinates $(\mathbf{q},\mathbf{p})$
on the cotangent bundle $T^*Q$.
If the magnetic form is exact
variational methods can be used
to prove existence of periodic solutions
as done for example
by Contreras \cite{con06} using
the {\bf Ma\~n\'e critical value}
\[
\inf_{\theta}\sup_{Q}H(\theta)\,,
\]
where the infimum is taken
over all primitive $1$-forms
$\theta$ of $\sigma$.
We remark that the Ma\~n\'e critical value
is always bounded from below
by $\max_{Q} V$.
If the magnetic form is non-exact
but lifts to an exact form to the universal cover of $Q$
it was shown by Merry \cite{mer10}
how to utilize the Ma\~n\'e critical value
of the lifted Hamiltonian
on the universal cover instead.
Finiteness of the Ma\~n\'e critical value
in that situation
is equivalent to the existence of a bounded
primitive with respect to the lifted metric.

Arnol'd \cite{arn61} and Novikov \cite{nov82,nosh81}
initiated the study of magnetic flows
in the context of symplectic topology.
Motivated by their existence results obtained by
Rabinowitz--Floer homology
Cieliebak--Frauenfelder--Paternain \cite{cfp10}
formulated a paradigm how a magnetic flow
should behave in terms of the Ma\~n\'e critical value
of the magnetic system.
In order to confirm the conjectured paradigm
for energies below the Ma\~n\'e critical value
in parts
we state the following theorem:

\begin{thmintr}
\label{thmintr:truncclasshammotion}
Let $Q$ be a closed hyperbolic surface,
i.e.\ a Riemannian surface of curvature $-1$.
Let $V$ be a Morse function on $Q$
that has a unique local maximum
which is required to be positive.
Let $v_0>0$ be a positive real number
such that $\{V\geq-v_0\}$
contains no critical point of $V$
other than the maximum.
Let $\sigma$ be a $2$-form on $Q$
that vanishes on the disc $\{V\geq-v_0\}$.
If
\[
v_0>\inf_{\theta}\sup_{\widetilde{Q}}\frac12|\theta|^2\,,
\]
where the infimum is taken over all
$C^3$-bounded primitive $1$-forms $\theta$
of the lift of $\sigma$ to the universal cover
$\widetilde{Q}$ that vanish on $\{\widetilde{V}\geq-v_0\}$
for the lifted potential $\widetilde{V}$,
then the equations of motion
of a charged particle on $Q$
under the influence of the magnetic field $\sigma$
and the presence of the potential $V$
have a non-constant periodic solution
of energy $H=0$ that is contractible in $\{V\leq0\}$.
\end{thmintr}

We will prove the theorem
in Section \ref{subsec:clahammagfie}.
Observe that the $2$-form $\sigma$
on the surface $Q$
is automatically closed but
not necessarily exact.
In the formulation of the theorem
we explicitly allow $\sigma$
to be non-exact.
As shown in Section \ref{subsec:truncthemagfiel}
there always exist $C^3$-bounded primitive $1$-forms $\theta$
on the hyperbolic upper half-plane
with the vanishing condition as stated,
so that the inequality in the theorem
indeed can be satisfied.
We remark
that a change of the metric
in the definition of the kinetic energy
results in a change of $v_0$
and that the Ma\~n\'e critical value
of the described magnetic system is positive.
In Example \ref{ex:reginflmagnontriv}
we describe a class of magnetic Hamiltonian systems
for which the trace of the solution found in
Theorem \ref{thmintr:truncclasshammotion}
intersects the region where the magnetic field does not vanish.

Periodic low-energy solutions
for Hamiltonian systems with
a non-vanishing magnetic field
were found by Schlenk \cite{sch06}.
If the magnetic form is symplectic
the reader is referred to the work of
Ginzburg--G\"urel \cite{gg09}.
The case where $Q$ equals the $2$-sphere
is discussed by Benedetti--Zehmisch in \cite{bz15}.
If $Q$ is a $2$-torus we refer to
Schlenk \cite{sch06} as well as to
Frauenfelder--Schlenk \cite{fsch07},
where periodic magnetic geodesics
(i.e.\ Hamiltonians with vanishing potentials)
are studied.
For more existence results of periodic solutions
in magnetic Hamiltonian mechanics
inspired by the work of
Ginzburg \cite{gin87}, Polterovich \cite{pol95},
and Ta\u\i manov \cite{tai92}
confer
\cite{aabmt17,ammp17,amp15,af07,ab16,ab15,bf11,bz15,mer11,sch06}
and the citations therein.

A higher-dimensional analogue to
Theorem \ref{thmintr:truncclasshammotion} is:

\begin{thmintr}
\label{thmintr:hightruncclasshammotion}
Let $Q$ be a product of closed hyperbolic surfaces.
Assume that $Q$ admits a potential function $V$
together with a choice of regular value $-v_0<0$
and a closed magnetic $2$-form $\sigma$
satisfying the conditions
described in Theorem \ref{thmintr:truncclasshammotion}.
In addition, assume that $\sigma$
is cohomologous to a $\R$-linear combination
of the area forms corresponding to the factors of $Q$.
Then the magnetic flow of the Hamiltonian $H$
carries a non-constant periodic solution
of zero energy that is contractible in $\{V\leq0\}$.
\end{thmintr}

In fact,
the proof given
in Section \ref{subsec:clahammagfie} shows
that Theorem \ref{thmintr:hightruncclasshammotion}
holds for any closed Riemannian manifold
$(Q,h)$ that carries a $2$-form $\sigma$
that is contained in a cohomology class
that is a linear combination of classes
represented by $2$-forms
admitting a $C^3$-bounded primitive
on the universal cover of $Q$.
For example one could take additional products
with closed Riemannian manifolds
in the formulation of
Theorem \ref{thmintr:hightruncclasshammotion}.

In order to prove the theorems
we consider the characteristics on
the energy surface $M$ of the Hamiltonian,
which is contained in $T^*Q$
equipped with the Liouville symplectic structure
twisted by the pull back of the magnetic form,
see \cite[Section 2.3]{wz16} and Section \ref{subsec:magensur}.
The main observation due to
Cieliebak--Frauenfelder--Paternain \cite{cfp10}
is that the energy surface $M'$ of the lifted Hamiltonian
on the universal cover of $T^*Q$ is of contact type
uniformly.
In fact,
the energy surface $M$ is virtually contact
in the sense of
Cieliebak--Frauenfelder--Paternain \cite{cfp10},
i.e.\ $M'$ admits a contact form $\alpha$
that is the restriction of a suitable primitive
of the twisted symplectic form,
cf.\ \cite[Section 2.1]{wz16} and Section \ref{subsec:virtcontstr}.
Due to the presence of a non-exact magnetic form
the energy surface $M'$ is typically non-compact,
see \cite[Theorem 1.2]{wz16}.

With the help of holomorphic curves
we will find periodic orbits on $M'$.
In the work of
Cieliebak--Frauenfelder--Paternain \cite{cfp10}
a variant of Rabinowitz--Floer homology
on the universal cover of $T^*Q$ is used.
In this article we will utilize
the fundamental work of Hofer \cite{hof93}
and study holomorphic curves in the symplectisation
of the contact manifold $(M',\alpha)$.
The Weinstein handle body structure
of the sublevel set of the Hamiltonian
is used in order to find germs of holomorphic disc fillings
as in the work of
Geiges--Zehmisch \cite{gz16a,gz16b} and
Ghiggini--Niederkr\"uger--Wendl \cite{gnw16},
see Section \ref{subsec:germofhlomdiscs}.

To handle the non-compactness of $M'$
local compactness properties
of holomorphic curves have to be ensured.
For holomorphic curves uniformly close
to the zero section
in the symplectisation of $(M',\alpha)$
the required uniform lower and upper bounds
on the contact form $\alpha$ suffice to
guarantee a tame geometry,
see Section \ref{sec:atamegeom}.
This leads to monotonicity type estimates
in Section \ref{sec:holomdiscs}.
In order to mimic Hofer's \cite{hof93}
asymptotic analysis we will use the deck transformation 
group of $\R\times M'$ which acts by isometries,
see Section \ref{sec:comp}.
The action on the tame structure in the $\R$-direction
is controlled by the use of the Hofer energy,
see \cite{hof93}.
The action on the contact form $\alpha$
in the direction of $M'$
is controlled by an Arzel\`a--Ascoli type argument
that requires
higher order bounds on the contact form,
see Section \ref{sec:hordboundsonprim}.
The upshot is that the holomorphic analysis developed in this paper
allows a reduction in the search of periodic orbits
to finding bounded primitives of higher order
on non-compact covering spaces.
Examples of how this principle works
are formulated in Section \ref{subsec:germofhlomdiscs}.
Furthermore,
it gives a method of how to find periodic Reeb orbits
on non-compact energy surfaces
that is different from the method in the work
of Suhr--Zehmisch \cite{sz16}.


\section{A tame geometry\label{sec:atamegeom}}

A virtually contact structure defines a bounded geometry
on the corresponding covering space.
In this section we describe the relation to the geometry
which is defined by the contact form on the covering space
and the induced complex structure
on the contact plane distribution.


\subsection{A virtually contact structure\label{subsec:virtcontstr}}

Let $M$ be a closed connected manifold
of dimension $2n-1$ for $n\geq2$.
Let $\omega$ be an odd-dimensional symplectic
form on $M$, i.e.\ a closed $2$-form
whose kernel is a $1$-dimensional distribution.
We assume that $(M,\omega)$ is virtually contact
in the sense of \cite{wz16}.
This means that we can choose a virtually contact structure
\[
\big(\pi\co M'\ra M, \alpha, \omega, g\big)
\]
on $(M,\omega)$,
where $\pi$ is a covering of $M$,
$\alpha$ is a contact form on the covering space $M'$
such that $\pi^*\omega=\rmd\alpha$,
and $g$ is a Riemannian metric on $M$,
whose lift $\pi^*g$ to $M'$ is denoted by $g'$.
By definition of a virtually contact structure
the primitive $\alpha$ is bounded with respect
to the norm $|\,.\,|_{(g')^{\flat}}$
of the dual metric $(g')^{\flat}$ of $g'$.
Moreover, there exists a constant $c>0$
such that for all $v\in\ker\rmd\alpha$
the following lower bound holds true:
\begin{equation}
 \label{lowerbound}\tag{LB}
 |\alpha(v)|\geq c|v|_{g'}.
\end{equation}
In addition,
we assume that the chosen
virtually contact structure is {\bf non-trivial},
i.e.\ that $\omega$ is not the exterior differential 
of a contact form on $M$.

Observe that the characteristic line bundle $\ker\omega$
of the odd-symplectic manifold $(M,\omega)$
is orientable precisely if $M$ is orientable.
In this situation we orient $M$, resp., $\ker\omega$
by the contact form $\alpha$ on $M'$ requiring
that the covering map $\pi$ preserves the orientation.


\subsection{A characterization\label{subsec:achar}}

The virtually contact property of the
odd-dimensional symplectic manifold $(M,\omega)$
can be formulated in a slightly different language.
We denote by $\xi=\ker\alpha$
the contact structure on $M'$
defined by the contact form primitive
$\alpha$ of $\pi^*\omega$.
The Reeb vector field of $\alpha$
is denoted by $R$
and defines a splitting $\R R\oplus\xi$ of $TM'$.
Moreover, the metric $g'$
defines an orthogonal splitting $\xi^{\perp}\oplus\xi$
of $TM'$
and the orthogonal projection onto $\xi^{\perp}$
is denoted by $\proj_{\perp}$.

With these notions understood
we formulate the geometric requirements
on $\alpha$,
which are given in terms of $g'$,
equivalently as follows:
There exist constants $c>0$ and $C>0$
such that
\[
|R|_{g'}\leq\frac{1}{c}
\quad\text{and}\quad
\frac{1}{C}\leq|\proj_{\perp}\!R\,|_{g'}.
\]
The constant $c$ corresponds to the one given in
Section \ref{subsec:virtcontstr}.
The constant $C$ can be taken to be $\sup_{M'}|\alpha|_{(g')^{\flat}}$,
where $(g')^{\flat}$ is the dual metric of $g'$.
We remark that $|\alpha|_{(g')^{\flat}}$ is equal to
the operator norm $\lVert\alpha\rVert$
of the linear form $\alpha$ on $TM'$
with respect to $g'$,
where the norm $\lVert\alpha\rVert$ is taken pointwise.


\subsection{The induced complex structure\label{subsec:theindcomstr}}

The equation
\[
\rmd\alpha=g'\big(\Phi(\,.\,),\,.\,\big)
\qquad\text{on}\quad\xi
\]
determines a skew adjoint vector bundle isomorphism
$\Phi\co\xi\ra\xi$ whose square multiplied by $-1$
is self adjoint and positive definite.
By \cite[Proposition 1.2.4]{hum97}
the bundle endomorphism given by
\[
j:=\Phi\circ\big(\sqrt{-\Phi^2}\big)^{-1}
\]
is a complex structure on $\xi$
that is compatible with $\rmd\alpha$,
i.e.\
\[
g_j:=\rmd\alpha(\,.\,,j\,.\,)
\]
defines a bundle metric on $\xi$.

\begin{lem}
\label{unifequiv1}
 The norm $|\,.\,|_j$ induced by $g_j$
 and the restriction of the norm $|\,.\,|_{g'}$
 to $\xi$ are uniformly equivalent,
 i.e.\ there exist constants $c_1,c_2>0$
 such that
 \[
 \frac{1}{c_1}|\,.\,|_{g'}
 \leq|\,.\,|_j\leq
 c_2|\,.\,|_{g'}
 \]
 on $\xi$.
\end{lem}

\begin{proof}
  Denote by $\lVert\Phi^{-1}\rVert$ and $\lVert\Phi\rVert$
  pointwise operator norms with respect to $g'$,
  which equal $1$ over the smallest resp.\ the largest
  pointwise eigenvalue of $\sqrt{-\Phi^2}$.
  Observe that
  \[
  g_j=g'\Big(\sqrt{-\Phi^2}(\,.\,),\,.\,\Big)
  \qquad\text{on}\quad\xi\,,
  \]
  so that
  \[
  \frac{1}{\sqrt{\lVert\Phi^{-1}\rVert}}\;|\,.\,|_{g'}
  \leq
  |\,.\,|_j
  \leq
  \sqrt{\lVert\Phi\rVert}\;|\,.\,|_{g'}
  \qquad\text{on}\quad\xi\,.
  \]
  Therefore,
  we have to show
  that the eigenvalues of $-\Phi^2$
  are uniformly positive and uniformly bounded.
  
  Let $\Omega\co TM\ra T^*M$
  be the linear bundle map given by
  $v\mapsto i_v\omega$.
  The characteristic line bundle of $(M,\omega)$
  is equal to $\ker\Omega$.
  For a constant $c_0>0$
  we define the convex subbundle
  $C$ of $TM$ to be the set of all tangent vectors
  $v\in TM$ whose angle with $\ker\Omega$
  with respect to $g$
  is greater or equal than $c_0$.
  In view of the alternative characterization
  of the virtually contact property given
  in Section \ref{subsec:achar}
  we can choose $c_0$ so small
  such that
  \[
  T\pi(\xi)\subset C\,.
  \]
  We consider the map
  \[
  TM\ni v\longmapsto|\Omega(v)|_{g^{\flat}}\in[0,\infty)\,,
  \]
  where $g^{\flat}$ denotes the dual metric of $g$.
  This map is uniformly bounded
  from above and away from zero
  on the compact set $C\cap STM$,
  where $STM$ denotes the unit tangent bundle
  of $M$ with respect to $g$.
  We define a cone like subbundle $C'$ of $TM'$
  to be the preimage of $C$ under $T\pi$
  with respect to $g'$
  that contains $\xi$
  such that
  \[
  C'\cap STM'\ni v\longmapsto|\Omega'(v)|_{(g')^{\flat}}\in(0,\infty)
  \]
  has uniform upper and lower bounds,
  where $\Omega'\co TM'\ra T^*M'$
  is the corresponding map
  $v\mapsto i_v\rmd\alpha$.
  Because of
  \[
  |\Omega'(v)|_{(g')^{\flat}}=|\Phi(v)|_{g'}
  \]
  for all $v\in\xi$
  we obtain that
  the map
  \[
  \xi\cap STM'\ni v\longmapsto|\Phi(v)|_{g'}\in(0,\infty)
  \]
  is uniformly bounded from above and away from zero.
  With
  \[
  \big|\Phi^2(v)\big|_{g'}=
  |\Phi(v)|_{g'}\;
  \left|
  \Phi
  \left(
  \frac{\Phi(v)}{\;\;|\Phi(v)|_{g'}}
  \right)
  \right|_{g'}
  \]
  we get similar bounds for
  \[
  \xi\cap STM'\ni v\longmapsto\big|\Phi^2(v)\big|_{g'}\in(0,\infty)\,.
  \]
  Inserting all possible eigenvectors $v\in\xi$
  of unit length with respect to $g'$
  we see that all eigenvalues of $-\Phi^2$
  are uniformly positive and uniformly bounded.
  This proves the lemma.
\end{proof}


\subsection{Bounding the geometry\label{subsec:boundthegeom}}

On the covering space $M'$
we define a second Riemannian metric
by setting
\[
g_{\alpha}=\alpha\otimes\alpha+g_j
\]
with respect to the splitting $\R R\oplus\xi$.
We denote the projection of $TM'$ onto $\xi$
along the Reeb vector field by $\pi_{\xi}$.

\begin{lem}
\label{unifequiv2}
 The norm $|\,.\,|_{\alpha}$ induced by $g_{\alpha}$
 and the norm $|\,.\,|_{g'}$
 are uniformly equivalent on $M'$,
 i.e.\ there exist constants $c_1,c_2>0$
 such that
 \[
 \frac{1}{c_1}|\,.\,|_{g'}
 \leq|\,.\,|_{\alpha}\leq
 c_2|\,.\,|_{g'}\;.
 \]
\end{lem}

\begin{proof}
 We begin with the first inequality:
 Decompose any given tangent vector
 $v\in TM'$ into $v=v^1R+Y$
 with respect to $\R R\oplus\xi$.
 Then
 \[
 |v|^2_{\alpha}=\big|\alpha(v^1R)\big|^2+|Y|^2_j\,.
 \]
 With \eqref{lowerbound} and Lemma \ref{unifequiv1}
 we obtain the lower estimate
 \[
 |v|^2_{\alpha}
 \geq
 \min\!\left(c^2,\frac{1}{c^2_1}\right)\cdot
 \Big(|v^1R|^2_{g'}+|Y|^2_{g'}\Big)\,.
 \]
 Using the parallelogram identity
 $2\big(|x|^2+|y|^2\big)=|x+y|^2+|x-y|^2$
 this leads to
 \[
 |v|^2_{\alpha}
 \geq
 \frac12\min\!\left(c^2,\frac{1}{c^2_1}\right)
 \cdot |v|^2_{g'}
 \]
 proving the first inequality.
 
 In order to show the second inequality observe
 that
 \[
 |v|^2_{\alpha}=\big|\alpha(v)\big|^2+|\pi_{\xi}v|^2_j\,.
 \]
 Using the boundedness of $\alpha$ and Lemma \ref{unifequiv1}
 this gives
 \[
 |v|^2_{\alpha}
 \leq
 \max\!\left(C^2,c^2_2\lVert\pi_{\xi}\rVert^2\right)\cdot|v|^2_{g'}\,.
 \]
 It remains to prove finiteness
 of the supremum of all pointwise operator norms
 $\lVert\pi_{\xi}\rVert$
 with respect to $g'$,
 where the supremum is taken over all points of  $M'$.
 For $p\in M'$
 let $u$ be a unit tangent vector with respect to $g'$
 that is either contained in $\xi$ if $R_p\in\xi^{\perp}$ or
 lies in the span of $\R R$ and $\xi^{\perp}$ at $p$
 being perpendicular to $R_p$
 and pointing in the same co-orientation direction of $\xi$
 as $R_p$.
 The norm $\lVert\pi_{\xi}|_p\rVert$
 is the length $|Y|_{g'}$ of the vector
 $Y=\pi_{\xi}u$ in $T_pM'$.
 Consequently,
 \[
 \lVert\pi_{\xi}|_p\rVert=
 \frac{1}{\sin\measuredangle_{g'}(R_p,\xi_p)}\,,
 \]
 which is uniformly bounded
 as the angle $\measuredangle_{g'}(R_p,\xi_p)$
 between $R_p$ and $\xi_p$
 stays uniformly away from zero
 according to the alternative formulation
 of the virtually contact property
 given in Section \ref{subsec:achar}.
 In other words,
 the supremum of all $\lVert\pi_{\xi}|_p\rVert$,
 $p\in M'$, is finite proving the second inequality.
\end{proof}


\subsection{Length and area\label{subsec:landa}}

The norms of $g_{\alpha}$ and $g'$
are uniformly equivalent by Lemma \ref{unifequiv2}
so that we can formulate isoperimetric type inequalities
with respect to either metric.

By definition the metric $g'$ is locally isometric to $g$
via the covering $\pi\co M'\ra M$.
The compactness of $M$ implies
that $g'$ is of bounded geometry.
In particular,
the absolute values of the sectional curvature of $g'$
are uniformly bounded
and the injectivity radius of $g'$
is uniformly bounded away from zero
by, say, $2i_0>0$.
Further,
the metric $g'$ is complete.
Moreover,
for all $p\in M'$ the exponential map
$\exp_p$ is defined for all tangent vectors
$v\in T_pM'$ of length $|v|_{g'}<i_0$
inducing a diffeomorphism
\[
T_pM'\supset B_{i_0}\!(0)\lra B_{i_0}\!(p)\subset M'
\]
onto the geodesic ball $B_{i_0}\!(p)$ with respect to $g'$.
Denote by $\EE_p$ the restriction of $\exp_p$ to $B_{i_0}\!(0)$.
By \cite[p.~318]{pet06} the linearization of $\EE_p$
and its inverse $\EE^{-1}_p$ are uniformly bounded,
i.e.\ there exists a constant $C>0$ such that for all $p\in M'$
\[
\lVert T\EE_p\rVert\,,\; \lVert T\EE^{-1}_p\rVert<C\,,
\]
where the operator norm is taken pointwise with respect to $g'$.

We will use this to formulate an isoperimetric inequality
for smooth loops that are contained in a geodesic ball of radius $i_0$.
Let $c\co\R\ra M'$ be a $2\pi$-periodic map 
with image in $B_{i_0}\!\big(c(0)\big)$.
Define a loop $X$ of tangent vectors in $T_{c(0)}M'$ via
$\exp_{c(0)}\!X(\theta)=c(\theta)$.
This defines a disc map $f_c\co D^2\ra M'$ via
\[
f_c(r\rme^{\rmi\theta})=\exp_{c(0)}\!\big(rX(\theta)\big)\,,
\]
where we use polar coordinates $z=r\rme^{\rmi\theta}$
on the closed unit disc $D^2$.
In view of the above first order bounds
on $\EE_{c(0)}$ and $\EE^{-1}_{c(0)}$
we obtain
\begin{equation}
\label{rderivative}
\big|\partial_rf_c(r\rme^{\rmi\theta})\big|_{g'}
\leq\frac{C}{2}\length_{g'}(c)\;,
\end{equation}
where
\[
\length_{g'}(c)=\int_0^{2\pi}|\dot{c}(\theta)|_{g'}\rmd\theta\;,
\]
and
\begin{equation}
\label{thetaderivative}
\big|\partial_{\theta}f_c(r\rme^{\rmi\theta})\big|_{g'}
\leq C^2|\dot{c}(\theta)|_{g'}\;.
\end{equation}
Therefore,
we can estimate the area
\[
\Area_{g'}\!\big(f_c(D^2)\big)=
\int_{(0,1)\times(0,2\pi)}
\sqrt{\det(f^*_cg')_{ij}}
\;\;\rmd r\wedge\rmd\theta
\]
of the disc $f_c(D^2)$
\[
\Area_{g'}\!\big(f_c(D^2)\big)
\leq
\frac{C^3}{2}\Big(\length_{g'}(c)\Big)^2\,.
\]
Observe,
that a similar inequality holds
if area and length are measured
with respect to the metric $g_{\alpha}$.


\section{Holomorphic discs\label{sec:holomdiscs}}

Holomorphic discs in symplectisations
$\R\times M'$
cannot only escape to $-\infty$ in the $\R$-direction
but they need $C^0$-control in the $M'$-directions too.
In this section
we will use monotonicity type phenomena
of holomorphic discs
to obtain distance estimates in $M'$-directions
in terms of symplectic energy and $\R$-distance bounds.


\subsection{An almost complex structure\label{subsec:almostcpxstr}}

On $\R\times M'$ we define an almost complex structure $J$
that is translation invariant, restricts to $j$ on $\xi$,
and sends $\partial_t$ to $R$ denoting by $t$ the
$\R$-coordinate.
We consider a holomorphic disc,
which is a smooth map
$u\co\D\ra\R\times M'$ such that
$Tu\circ\rmi=J(u)\circ Tu$,
where $\D$ denotes the closed unit disc in $\C$
equipped with the complex structure $\rmi$.
We will assume
that the following boundary condition is satisfied
$u(\partial\D)\subset\{0\}\times M'$.
Writing $u=(a,f)$ holomorphicity of $u$
can be expressed as
\[
\begin{cases}
\;\,-\rmd a\circ\rmi=f^*\alpha\,,\\
\pi_{\xi}Tf\circ\rmi=j(f)\circ\pi_{\xi}Tf\,.
\end{cases}
\]
In particular,
$a\co\D\ra\R$ is a subharmonic function due to the choice of $j$.
The maximum principle implies that $u(\D)$
is contained in $(-\infty,0]\times M'$.
Moreover,
\[
u^*(\rmd t\wedge\alpha)=
\big(a^2_x+a^2_y\big)\,\rmd x\wedge\rmd y
\]
and
\[
f^*\rmd\alpha=
\frac12\Big(|f_x|^2_{g_j}+|f_y|^2_{g_j}\Big)\,\rmd x\wedge\rmd y\;.
\]
Observe that both expressions are non-negative.


\subsection{Area growth\label{subsec:areagrowth}}

For $p\in M'$ and $t\in(0,i_0)$
we consider the solid cylinder
$\R\times B_t(p)$ over the open
geodesic ball $B_t(p)$ with respect to $g'$
and denote the preimage
of the intersection with the holomorphic disc by
\[
G_t=
u^{-1}\big(\R\times B_t(p)\big)=
f^{-1}(B_t(p)\big)\,.
\]
We assume that
\[
f(\partial\D)\subset M'\setminus B_{i_0}\!(p)
\]
so that $G_t$ is disjoint
from the boundary $\partial\D$ of $\D$.
The radial distance function of $g'$ at $p$
is denoted by
\[
r\co B_{i_0}\!(p)\lra[0,i_0)\quad\text{;}\qquad
x\longmapsto
\dist_{g'}(p,x)\;.
\]
Notice,
that using the Gau{\ss}-lemma 
the pointwise operator norm of
$Tr$ with respect to $g'$
equals $\lVert Tr\lVert=1$,
see \cite[Lemma 6.12]{pet06}.
The restriction of the radial distance function
to the holomorphic disc is denoted by
\[
F\co G_{i_0}\lra[0,i_0)\quad\text{;}\qquad
z\longmapsto
r\big(f(z)\big)\;.
\]
With this notation introduced we see
that $\partial G_t=F^{-1}(t)$.
Denote by $\Reg\subset(0,i_0)$
the set of regular values of $F$
that are not contained in the image
$r\big(\{\pi_{\xi} Tf=0\}\big)$,
which is a finite set,
see \cite[Lemma 7]{gz16b}.
Observe that $f$ has no critical points
on $F^{-1}(\Reg)$.
Let $h$ be a Riemannian metric on $F^{-1}(\Reg)\subset\C$
for which there exists a universal constant
$C_0>0$ that is independent of the choice of $p\in M'$
satisfying
\[
\big|\grad_hF\big|_h\leq\frac{1}{C_0}\;.
\]
In order to find such a universal constant $C_0$
we remark that
\[
\big|\!\grad_hF\big|_h^2=
\rmd F\big(\!\grad_hF\big)\;.
\]
Since the Gau{\ss} lemma implies $\lVert Tr\lVert=1$
we see that for all $v\in TG_t$
\begin{equation}
\label{choosev}
|\rmd F(v)|_h\leq|Tf(v)|_{g'}
\end{equation}
applying the chain rule to $F=r\circ f$.

With the co-area type arguments
in the proof of the monotonicity lemma
given in \cite[p.~27/28]{hum97} we obtain:

\begin{lem}
 \label{derivativeofarea}
 For all $t\in\Reg$ the $t$-derivative of the area of $G_t$
 exists and satisfies
 \[
 \big(\!\Area_h(G_t)\big)'
 \geq C_0\,\length_h(\partial G_t)\;.
 \]
\end{lem}


\subsection{Symplectisation\label{subsec:symplectisation}}

Denote by $\TT$ the set of all smooth
strictly increasing functions
$\tau\co (-\infty,0]\ra[0,1]$ with $\tau(0)=1$.
Any $\tau\in\TT$ defines a symplectic form
\[
\rmd(\tau\alpha)=\tau'\rmd t\wedge\alpha+\tau\rmd\alpha
\]
on $\R\times M'$.
The almost complex structure $J$
is compatible with $\rmd(\tau\alpha)$
defining a metric
\[
g_{\tau}=\rmd(\tau\alpha)(\,.\,,J\,.\,)=
\tau'\big(\rmd t\otimes\rmd t+\alpha\otimes\alpha\big)+
\tau g_j\;.
\]
The holomorphicity of $u$ implies
that $u^*g_{\tau}$ defines a conformal metric
$h$ on $F^{-1}(\Reg)$.
Inserting
\[
v=\frac{\grad_hF}{\;\big|\!\grad_hF\big|_h}
\]
into \eqref{choosev}
with that choice of $h$
yields the universal constant $C_0$
required in Lemma \ref{derivativeofarea}
via
\[
\big|\!\grad_hF\big|_h\leq
c_1\,\max_{a(\D)}
\left(\frac{1}{\,\tau'},\frac{1}{\tau}\right)\;,
\]
where $c_1$ is a constant from Lemma \ref{unifequiv2}.
Notice, that the right hand side
depends on the holomorphic disc $u=(a,f)$.


\subsection{An isoperimetric inequality\label{subsec:isoperiineq}}

We continue the discussion from
Section \ref{subsec:areagrowth}.
For $\tau\in\TT$ we will study the area
\[
A(t)=\int_{G_t}u^*\rmd(\tau\alpha)\,,
\]
$t\in[0,i_0)$,
which is cut out by the solid cylinder
$\R\times B_t(p)$ about the
$g'$-geodesic ball $B_t(p)$ in $M'$.
As in Section \ref{subsec:areagrowth}
we only allow those geodesic balls
that do not hit the boundary $f(\partial\D)$.
The length of the maybe disconnected
curves $u(\partial G_t)$,
to which we will compare the area $A(t)$,
is measured with respect to the metric
\[
g_0=\rmd t\otimes\rmd t+g_{\alpha}\,.
\]

\begin{lem}
\label{lem:isoperineqinsympl}
 There exists a positive constant $c_3$,
 which only depends on the geometry of $(M',g')$,
 such that
 \[
 A(t)\leq
 c_3\,\Big(1+\max_{G_t}\big(\tau'(a)\big)\Big)\,
 \Big(\length_{g_0}\!\big(u(\partial G_t)\big)\Big)^2\;,
 \]
 where
 $\length_{g_0}\!\big(u(\partial G_t)\big)$
 is the sum of the lengths of the components.
\end{lem}

\begin{proof}
 Let $N$ be the number of boundary components of $G_t$.
 For the $\ell$-th component of $\partial G_t$ define a disc map
 $f_{\ell}\co D^2\ra M'$ via $g'$-geodesics along the
 $\ell$-th component of
 the curve $f(\partial G_t)$
 as examined in Section \ref{subsec:landa}.
 Similarly,
 disc maps $a_{\ell}\co D^2\ra\R$
 are defined via convex interpolations
 along the components of $a(\partial G_t)$.
 Choosing orientations appropriately
 Stokes' theorem implies
 \[
 A(t)=
 \int_{\partial G_t}u^*(\tau\alpha)=
 \sum_{\ell=1}^N\int_{D^2}(a_{\ell},f_{\ell})^*\rmd(\tau\alpha)\;.
 \]
 The integrand decomposes into
 \[
 \tau'(a_{\ell})\rmd a_{\ell}\wedge f_{\ell}^*\alpha+
 \tau(a_{\ell})f_{\ell}^*\rmd\alpha\;.
 \]
 We will tread the summands separately.
 Beginning with the second,
 which is bounded by $|f_{\ell}^*\rmd\alpha|$
 because of $\tau\leq1$,
 we find
 \[
 \int_{D^2}|f_{\ell}^*\rmd\alpha|
 \leq
 \frac{c_2^2C^3}{2}
 \Big(\length_{g'}\!\big(f_{\ell}(\partial D^2)\big)\Big)^2\
 \]
 for all $\ell=1,\ldots,N$
 using Lemma \ref{unifequiv2}
 and equations \eqref{rderivative} and \eqref{thetaderivative}
 from Section \ref{subsec:landa}.
 Hence,
 the sum of the $(a_{\ell},f_{\ell})^*(\tau\rmd\alpha)$--integrals
 is estimated by
 \[
 c_4\,\Big(\length_{g'}\!\big(f(\partial G_t)\big)\Big)^2
 \leq
 c_1^2c_4\,\Big(\length_{g_{\alpha}}\!\big(f(\partial G_t)\big)\Big)^2
 \]
 setting $c_4=c_2^2C^3/2$ and using Lemma \ref{unifequiv2} again.
 Similarly,
 invoking the boundedness of $\alpha$
 and equations \eqref{rderivative} and \eqref{thetaderivative} again
 we obtain
 \[
 \int_{D^2}\tau'(a_{\ell})\rmd a_{\ell}\wedge f_{\ell}^*\alpha
 \leq
 c_5\,\max_{G_t}\big(\tau'(a)\big)\,
 \Big(\length_{g_0}\!\big((a_{\ell},f_{\ell})(\partial D^2)\big)\Big)^2\
 \]
 for a positive constant $c_5$.
 Therefore,
 the sum of the $(a_{\ell},f_{\ell})^*(\tau'\rmd t\wedge\alpha)$--integrals
 is bounded by
 \[
 c_5\,\max_{G_t}\big(\tau'(a)\big)\,
 \Big(\length_{g_0}\!\big(u(\partial G_t)\big)\Big)^2\;.
 \]
 Combining both estimates proves the claim.
\end{proof}


\subsection{Monotonicity\label{subsec:monotonicity}}

For $h$ being induced by $u^*g_{\tau}$
as in Section \ref{subsec:symplectisation},
Lemma \ref{derivativeofarea} implies
that there exists a positive constant $c_6$
such that
\[
A'(t)\geq
\frac{c_6}{\left(
\max_{a(\D)}
\left(\frac{1}{\,\tau'},\frac{1}{\tau}\right)
\right)^2}
\length_{g_0}\!\big(u(\partial G_t)\big)
\]
for all $t\in\Reg$.
Combined with Lemma \ref{lem:isoperineqinsympl},
which says that
\[
\length_{g_0}\!\big(u(\partial G_t)\big)
\geq
\sqrt{
\frac{c_3^{-1}}
{1+\max_{a(\D)}(\tau')}
}\,
\sqrt{A(t)}
\;,
\]
we obtain
\begin{equation}
\label{diffequforarea}
A'(t)\geq2\,\mathrm{m}\,\sqrt{A(t)}
\;,
\end{equation}
for all $t\in\Reg$,
where
\[
\mathrm{m}=
\mathrm{m}\big(\tau(a)\big):=
\frac{c_6}{2\sqrt{c_3}}\;
\frac{1}{\left(
\max_{a(\D)}
\left(\frac{1}{\,\tau'},\frac{1}{\tau}\right)
\right)^2}\;
\sqrt{
\frac{1}
{1+\max_{a(\D)}(\tau')}
}
\;.
\]
With the reasoning on \cite[p.~28]{hum97}
equation \eqref{diffequforarea}
implies the following monotonicity lemma
in symplectisations as described in the present context:

\begin{prop}
\label{prop:monotonocitylemma}
 Let $u=(a,f)$ be a holomorphic disc map
 that sends $(\D,\partial\D)$ into
 $(\R\times M',\{0\}\times M')$.
 Let $p\in M'$ be a point on $f(\D)$
 such that the $g'$-geodesic ball $B_{i_0}\!(p)$
 and the boundary curve $f(\partial\D)$
 are disjoint.
 Then the area functional
 \[
 A(t)=\int_{f^{-1}\big(B_t(p)\big)}u^*\rmd(\tau\alpha)
 \]
 satisfies
 \[A(t)\geq\mathrm{m}^2\,t^2\]
 for all $t\in[0,i_0)$,
 where the constant
 $\mathrm{m}=\mathrm{m}\big(\tau(a)\big)$
 is positive and
 depends on the $\R$-coordinate
 of the holomorphic disc $u=(a,f)$.
\end{prop}


\subsection{A distance estimate\label{subsec:adistestim}}

If we take $\tau(t)=\rme^t$
we obtain for the monotonicity constant $\mathrm{m}$
in Proposition \ref{prop:monotonocitylemma}
\[
\mathrm{m}(\rme^a)=c_7\,\rme^{-2\max_{\D}|a|}
\]
for a positive constant $c_7$.
We will use this in order to estimate the maximal distance
\[
\dist_{g'}\!\big(L,f(\D)\big)=
\sup_{f(\D)}\dist_{g'}(L,\,.\,)
\]
between a maximally $J$-totally real submanifold $L\subset M'$
that has compact closure and
the $M'$-part of a holomorphic disc $f(\D)$
that has boundary $f(\partial\D)$ on $L$.

\begin{prop}
 \label{prop:distestihofenerg}
 There exists a positive constant $K$
 that only depends on the geometry of $(M',g')$
 such that for all holomorphic discs
 \[
 u=(a,f)\co
 (\D,\partial\D)\lra
 (\R\times M',\{0\}\times L),
 \]
 where $L\subset M'$ is a relatively compact
 maximally totally real submanifold with respect to $J$,
 the following estimate holds
 \[
 \dist_{g'}\!\big(L,f(\D)\big)\leq
 K\,\rme^{{4\max_{\D}|a|}}E(u)
 \]
 denoting by
 \[
 E(u)=\int_{\D}u^*\rmd(\rme^t\alpha)
 \]
 the symplectic energy of $u$ with respect to $\rmd(\rme^t\alpha)$.
\end{prop}

\begin{proof}
 There are unique $d_0\in[0,2i_0)$
 and $N\in\Z$ such that
 \[
 \dist_{g'}\!\big(L,f(\D)\big)=2Ni_0+d_0\;.
 \]
 We assume that $N\neq0$
 because the claim holds by Proposition \ref{prop:monotonocitylemma}.
 Choose points $p_1,\ldots,p_N$
 on $f(\D)$ such that $\dist_{g'}(L,\,.\,)$
 maps $B_{i_0}\!(p_{\ell})$ into the shifted intervall
 \[
 \dist_{g'}\!\big(L,f(\D)\big)-\big(2i_0(N-\ell+1),2i_0(N-\ell)\big)
 \]
 for all $\ell=1,\ldots,N$.
 By Proposition \ref{prop:monotonocitylemma}
 the symplectic $\rmd(\rme^t\alpha)$--energy
 of the intersection $u(\D)\cap\R\times B_{i_0}\!(p_{\ell})$
 is bounded from below as
 \[
 \int_{f^{-1}\big(B_{i_0}\!(p_{\ell})\big)}u^*\rmd(\rme^t\alpha)
 \geq
 c_7^2\,\rme^{-4\max_{\D}|a|}\,i_0^2
 \]
 for all $\ell=1,\ldots,N$.
 Because the preimages $f^{-1}\big(B_{i_0}\!(p_{\ell})\big)$
 are mutually disjoint taking the sum over all points
 $p_1,\ldots,p_N$ yields
 \[
 E(u)
 \geq
 c_7^2\,\rme^{-4\max_{\D}|a|}\,Ni_0^2\,.
 \]
 Combining this with
 \[
 2Ni_0+d_0<(N+1)2i_0\leq 4Ni_0
 \]
 yields
 \[
 \dist_{g'}\!\big(L,f(\D)\big)
 \leq\frac{4}{c_7^2}\,\frac{1}{i_0}\,\rme^{4\max_{\D}|a|}E(u)
 \]
 proving the claim.
\end{proof}

In particular,
families of holomorphic discs
with uniform energy bounds
that lie above a certain slice
$\{-R\}\times M'$, $R\gg1$,
have uniformly $C^0$-bounded
projections into $M'$ with respect to $g'$.
Therefore,
Gromov's compactness theorem
as formulated in \cite{fr08,fz15} applies:

\begin{cor}
\label{cor:lowerboundgromconv}
 Let
 \[
 u_{\nu}=(a_{\nu},f_{\nu})\co
 (\D,\partial\D)\lra
 (\R\times M',\{0\}\times L)
 \]
 be a sequence of holomorphic discs,
 where $L\subset M'$ is a relatively compact
 maximally totally real submanifold with respect to $J$.
 Assume that there exists a compact subset of $L$ 
 that contains all boundary curves $u_{\nu}(\partial\D)$.
 If
 \[
 \sup_{\nu\in\N}\max_{\D}|a_{\nu}|<\infty
 \qquad\text{and}\qquad
 \sup_{\nu\in\N}E(u_{\nu})<\infty\,,
 \]
 then $u_{\nu}$ has a subsequence
 that Gromov converges to a stable holomorphic disc.
\end{cor}


\section{Higher order bounds on primitives\label{sec:hordboundsonprim}}

The group of isometries acts
on virtually contact structures via pull back.
The induced action on the space of
holomorphic discs in symplectisations
$\R\times M'$ that have no uniform $\R$-distance bounds
will be part of the bubbling off analysis,
see Section \ref{sec:comp}.
In this section we will work out
the required compactness properties
on sequences of virtually contact structures
obtained by the group action.


\subsection{Higher order covariant derivative\label{subsec:hordcovderiv}}

We consider a connected Riemannian manifold $(M',g')$.
The covariant derivative of the Levi--Civita connection
is denoted by $\nabla$.
Let $\tau$ be a $(0,k)$-tensor, $k\in\N$, on $M'$.
Following \cite[p.~73]{ghl04} and \cite[p.~52]{klb95}
we define the covariant derivative $\nabla_{\!X}\tau$
of $\tau$ in direction of the vector field $X$ on $M'$
via the following formula:
\[
\big(\nabla_{\!X}\tau\big)(Y_1,\ldots,Y_k):=
X\big(\tau(Y_1,\ldots,Y_k)\big)-
\sum_{j=1}^k\tau\big(Y_1,\ldots,Y_{j-1},\nabla_{\!X}Y_j,Y_{j+1},\ldots,Y_k\big)\;,
\]
where $Y_1,\ldots,Y_k$ are test vector fields on $M'$.
Setting
\[
\nabla\tau(X,Y_1,\ldots,Y_k):=
\big(\nabla_{\!X}\tau\big)(Y_1,\ldots,Y_k)
\]
defines a $(0,k+1)$-tensor $\nabla\tau$ on $M'$.
For a $1$-form $\alpha$ on $M'$ we inductively define
the $k$--{\bf th covariant derivative} by
$\nabla^0\alpha=\alpha$ and
\[
\nabla^k\alpha=\nabla\big(\nabla^{k-1}\alpha\big)
\]
so that $\nabla^k\alpha$ is a $(0,k+1)$-tensor $\nabla\tau$ on $M'$.
The pointwise norm of $\nabla^k\alpha$ at $p\in M'$
is defined by
\[
\big|\nabla^k\alpha\big|_p:=
\sup\big|\nabla^k\alpha(v,w_1,\ldots,w_k)\big|\;,
\]
where the supremum is taken over all
tuples $(v,w_1,\ldots,w_k)$ of unit tangent vectors
of $(M',g')$ at $p$.
We set
\[
\big\|\nabla^k\alpha\big\|_{C^0}:=
\sup_{M'}\big|\nabla^k\alpha\big|\;,
\]
which is the supremum of $\nabla^k\alpha$
on the $(k+1)$--fold Whitney sum of the unit
tangent bundle $STM'$.
The $C^k$--{\bf norm} of $\alpha$ on $(M',g')$
is defined via
\[
\|\alpha\|_{C^k}:=
\sup_{\ell=0,1,\ldots,k}\big\|\nabla^{\ell}\alpha\big\|_{C^0}\;.
\]
We provide the space of smooth $1$-forms on $M'$
that are bounded in all $C^k$-norms
with the $C^{\infty}$-topology with respect to the sequence
of $C^k$-norms as defined.
Observe that convergence on the restrictions to
relatively compact open subsets of $M'$ is the same
as the convergence induced by the compact open topology
as introduced in \cite[Section 2.1]{hir94}.
In this situation we will simply speak about
convergence in the $C^{\infty}_{\loc}$-topology.

\begin{rem}
 \label{rem:ckforfunctions}
 Similarly, for a smooth function $f$ on $M'$
 one defines $\nabla^0f=f$ and
 \[
 \nabla^kf=\nabla\big(\nabla^{k-1}f\big)=\nabla^{k-1}\rmd f
 \]
 as well as
 \[
 \|f\|_{C^k}:=
 \sup_{\ell=0,1,\ldots,k}\big\|\nabla^{\ell}f\big\|_{C^0}=
 \sup\Big\{\|f\|_{C^0}, \big\|\rmd f\big\|_{C^{k-1}}\Big\}\;.
 \]
 The space of smooth functions on $M'$
 that have finite $C^k$-norm for all $k\in\N$
 is provided with the $C^{\infty}$-topology
 induced by the $C^k$-norms.
 The $C^{\infty}_{\loc}$-convergence
 is understood in the same manner as for $1$-forms.
\end{rem}

\begin{rem}
 \label{rem:behaviorunderisometries}
 Let $\varphi$ be an isometry of $(M',g')$.
 The generalized {\it theorema egregium}
 can be phrased as
 \[
 \varphi_*\big(\nabla_{\!X}Y\big)=
 \nabla_{\!\varphi_*X}\varphi_*Y
 \]
 for all vector fields $X$ and $Y$ on $M'$,
 cf.\ \cite[p.~183]{rosa13}.
 It follows that
 \[
 \varphi^*\big(\nabla\tau\big)=
 \nabla(\varphi^*\tau)
 \]
 on $(0,k)$-tensors $\tau$ on $M'$.
 Inductively, for all $1$-forms $\alpha$ on $M'$
 we get
 \[
 \varphi^*\big(\nabla^k\alpha\big)=
 \nabla^k(\varphi^*\alpha)\;.
 \]
 Because $\varphi$ induces a bundle isomorphism
 on $\bigoplus^{k+1}STM'$ we obtain
 \[
 \|\varphi^*\alpha\|_{C^k}=
 \|\alpha\|_{C^k}\;.
 \]
\end{rem}


\subsection{Local computations\label{subsec:loccomp}}

We continue the discussions from Section \ref{subsec:hordcovderiv}.
Let $x^1,\ldots,x^{2n-1}$ be local coordinates on $M'$.
For a given Riemannian metric $g'$
and a given $1$-form $\alpha$ we write
$g'=g_{ij}\rmd x^i\otimes\rmd x^j$
and $\alpha=\alpha_j\rmd x^j$.
Denoting the Christoffel symbols of the Levi--Civita
connection $\nabla$ of $g'$ by $\Gamma_{ij}^k$
we get $(\nabla\alpha)_{ij}=\alpha_{j,i}-\Gamma_{ij}^{\ell}\alpha_{\ell}$.
Inductively,
\[
\big(\nabla^{k+1}\alpha\big)_{ij_1\ldots j_{k+1}}=
\big(\nabla^k\alpha\big)_{j_1\ldots j_{k+1},i}-
\sum_{\ell=1}^{k+1}
\Gamma_{ij_{\ell}}^m
\big(\nabla^k\alpha\big)_{j_1\ldots m\ldots j_{k+1}}
\]
where the $m$ ist placed at the $\ell$-th position.

In order to estimate the tensor $\nabla^k \alpha$
we notice that for the symmetric positive definite
matrix $(g_{ij})_{ij}$ we find an orthogonal matrix
$A$ such that
$(g_{ij})_{ij}=ADA^T$,
where $D$ denotes the diagonal matrix
of all eigenvalues $\lambda_1\leq\ldots\leq\lambda_{2n-1}$
of $(g_{ij})_{ij}$.
Hence, writing $v=v^i\partial_i$ for a tangent vector
we get
$|v|_{g'}^2\geq\lambda_1|v^i|^2$
for all $i=1,\ldots,2n-1$.
Taking unit tangent vectors
$(v,w_1,\ldots,w_k)$ we finally get
\[
\big|\nabla^k\alpha(v,w_1,\ldots,w_k)\big|
\leq
\left(\frac{2n-1}{\sqrt{\lambda_1}}\right)^{k+1}
\max_{ij_1\ldots j_k}
\left\{\Big|\big(\nabla^k\alpha\big)_{ij_1\ldots j_k}\Big|\right\}
\;.
\]


\subsection{Uniform $C^{\infty}$-bounds -- 
An example\label{subsec:ucibanexample}}

Using the results from Section \ref{subsec:loccomp}
we consider the following example:
Denote by $H^+$ the open upper half-plane
$\{y>0\}$ provided with the standard hyperbolic metric.
Let $M'$ be $\R\times H^+$ provided with the product metric
\[
g'=
\rmd t\otimes\rmd t+
\frac{1}{y^2}\big(\rmd x\otimes\rmd x+\rmd y\otimes\rmd y\big)
\;.
\]
Labeling the coordinates $(t,x,y)\in\R\times H^+$ by
$(x^1,x^2,x^3)$ the Christoffel symbols read as
\[
\Gamma_{ij}^k=
\frac{1}{y}
\Big(
\delta_{k3}\big(\delta_{i2}\delta_{j2}-\delta_{i3}\delta_{j3}\big)
-\delta_{k2}\big(\delta_{i2}\delta_{j3}+\delta_{i3}\delta_{j2}\big)\Big)
\;,
\]
which can be brought in the form
\[
\Gamma_{ij}^k=
\frac{1}{y}\gamma_{ij}^k
\]
for constants $\gamma_{ij}^k$ that vanish
if at least one index equals $1$.
The covariant derivatives
of the contact form
\[
\alpha=
\rmd t+\frac1y\rmd x=
\left(\delta_{i1}+\frac1y\delta_{i2}\right)\rmd x^i
\]
with respect to the Levi--Civita connection $\nabla$ of $g'$
are given by
\[
(\nabla\alpha)_{ij}=
\frac{1}{y^2}\delta_{i2}\delta_{j3}=
\frac{1}{y^2}\Delta_{ij}
\]
and
\[
\big(\nabla^k\alpha\big)_{j_1\ldots j_{k+1}}=
\frac{1}{y^{k+1}}\Delta_{j_1\ldots j_{k+1}}
\]
as an induction shows,
where $\Delta_{ij}$ and $\Delta_{j_1\ldots j_{k+1}}$
are constants that vanish
provided that at least one index is equal to $1$.
In particular, in the non-vanishing case
we find for all $k\in\N$ a positive constant $c_k$
such that for all index tuples $(j_1,\ldots, j_{k+1})$
\[
\left|
\big(\nabla^k\alpha\big)_{j_1\ldots j_{k+1}}
\right|\leq
\frac{c_k}{y^{k+1}}
\;.
\]
Because the eigenvalues of the matrix
obtained by the metric coefficients
of the hyperbolic metric on $H^+$
are equal to $\lambda=\frac{1}{y^2}$ we obtain
$|\alpha|_{(g')^{\flat}}\leq2$ and
\[
\big|\nabla^k\alpha\big|_{(g')^{\flat}}\leq
c_k
\left(\frac{2}{\sqrt{\lambda}}\cdot\frac{1}{y}\right)^{k+1}=
2^{k+1}c_k=:C_k
\;.
\]
In other words,
for all $k\in\{0\}\cup\N$ we have that $\|\alpha\|_{C^k}$
is globally bounded.

\begin{rem}
\label{rem:cinftybounds}
 Let $\Sigma$ be a closed hyperbolic surface.
 On $M=S^1\times\Sigma$
 an odd-symplectic form $\omega$
 is given by the area form of $\Sigma$.
 The lift of $\omega$ to the universal covering
 has primitive $\alpha$ as described
 and defines a virtually contact structure
 that is $C^k$-bounded for all $k\in\N$.
 Replacing $\Sigma$ by a product of
 closed hyperbolic surfaces and $\omega$
 by the direct sum of the corresponding area forms
 one obtains examples of $C^{\infty}$-bounded
 virtually contact structures in all odd dimensions
 in a similar manner.
\end{rem}


\subsection{An Arzel\`a--Ascoli argument\label{subsec:acovering}}

We consider a covering $\pi\co M'\ra M$.
Let $g$ be a Riemannian metric and
$\omega$ be an odd-symplectic form both on $M$.
We set $g'=\pi^*g$ and
denote the pull back form along the covering map $\pi$
by $\omega'=\pi^*\omega$.
The group $G$ of deck transformations
of the covering $\pi$ acts by isometries and
by odd-symplectomorphisms on $(M,g',\omega')$,
i.e.\ $\varphi^*g'=g'$ and $\varphi^*\omega'=\omega'$
for all $\varphi\in G$.
We assume that $\omega'$ has a primitive $1$-form
$\alpha$ on $M'$ so that
\[
\omega'=\rmd\alpha\;.
\]
For a sequence $(\varphi_{\nu})_{\nu}\subset G$, $\nu\in\N$,
of deck transformations we define a sequence
\[
\alpha_{\nu}:=\varphi_{\nu}^*\alpha
\]
of $1$-forms on $M'$.
We remark that because all $\varphi_{\nu}$ are isometries
we get with Remark \ref{rem:behaviorunderisometries}
for all $k\in\N$
\[
\|\alpha_{\nu}\|_{C^k}=
\|\alpha\|_{C^k}
\]
and that because all $\varphi_{\nu}$ are odd-symplectic
the $1$-form $\alpha_{\nu}-\alpha$ is closed on $M'$.

\begin{prop}
 \label{prop:arzelaascoli}
 We assume that the base manifold $M$ is closed
 and that for all $k\in\N$ there exists $C_k>0$
 such that
 \[
 \|\alpha\|_{C^k}<C_k\;.
 \]
 Then $\alpha_{\nu}$ has a subsequence that
 converges in $C^{\infty}_{\loc}(M')$.
\end{prop}

\begin{proof}
 We first prove the statement
 under the additional assumption
 that the first de Rham cohomology group
 of the covering space $M'$ vanishes.
 Therefore,
 the $1$-form $\alpha_{\nu}-\alpha$ is
 exact for all $\nu$.
 Denoting the base point of $M'$ by $o$
 we get in fact that for all $\nu$ there exists
 a unique $f_{\nu}\in C^{\infty}(M')$
 such that
 \[
 f_{\nu}(o)=0
 \quad\text{and}\quad
 \rmd f_{\nu}=\alpha_{\nu}-\alpha
 \;.
 \]
 In particular, we get
 \[
 \big\|\rmd f_{\nu}\big\|_{C^k}\leq
 \|\alpha_{\nu}\|_{C^k}+\|\alpha\|_{C^k}<
 2C_k
 \]
 providing a uniform bound
 \[
 \sup_{\nu\in\N}\big\|\rmd f_{\nu}\big\|_{C^k}\leq
 2C_k\;.
 \]
 In order to prove the proposition
 it suffices to show that for a subsequence of $f_{\nu}$
 there exists $f\in C^{\infty}(M')$ such that
 $f_{\nu}\ra f$ in $C^{\infty}_{\loc}$ as $\nu$ tends to $\infty$.
 Indeed, this will imply $\rmd f_{\nu}\ra \rmd f$
 in $C^{\infty}_{\loc}$ and, therefore,
 \[
 \alpha_{\nu}=\rmd f_{\nu}+\alpha
 \lra
 \rmd f+\alpha=:\alpha_0
 \]
 in $C^{\infty}_{\loc}$ as $\nu$ tends to $\infty$.
 
 Because the base manifold $M$ is closed
 by \cite[2.91 and 2.105]{ghl04} $(M',g')$
 is geodesically complete.
 By the Hopf--Rinow theorem
 (see \cite[2.103 and 2.105]{ghl04})
 we find for any given point $p\in M'$
 a minimal geodesic $c$ in $(M',g')$ of unit speed
 that connects $p$ with the base point $o$.
 By the mean value theorem
 we find for all $\nu\in\N$ a real number $t_{\nu}$
 such that
 \[
 |f_{\nu}(p)|\leq
 \dist_{g'}(o,p)\,\big|T_{c(t_{\nu})}f_{\nu}\big|_{g'}
 \;.
 \]
 Denoting by $B_r$ the open $g'$-geodesic ball
 with center $o$ and radius $r>0$ this implies
 \[
 \|f_{\nu}\|_{C^0(B_r)}\leq
 r\,\big\|\rmd f_{\nu}\big\|_{C^0(B_r)}\leq
 2rC_0
 \]
 for all $\nu\in\N$, which results in a uniform bound
 \[
 \sup_{\nu\in\N}\|f_{\nu}\|_{C^0(B_r)}\leq
 2rC_0
 \;.
 \]
 This means that for all $r>0$
 the subset $\{f_{\nu}|_{B_r}\}$ of $C^0(B_r)$ is bounded.
 Similarly, replacing $o$ by any point $q\in B_r$
 we get for all $p,q\in B_r$
 \[
 \big|f_{\nu}(p)-f_{\nu}(q)\big|\leq
 \dist_{g'}(p,q)\,2C_0
 \;.
 \]
 Therefore, the subset $\{f_{\nu}|_{B_r}\}$
 of $C^0(B_r)$ is equicontinuous, i.e.\
 \[
 \sup_{\nu\in\N}
 \big|f_{\nu}(p)-f_{\nu}(q)\big|
 \lra 0
 \]
 as $\dist_{g'}(p,q)$ tends to zero for points $p,q$
 in $B_r$.
 Observe that by the Hopf--Rinow theorem
 the closure of $B_r$ is compact for all $r>0$ and
 the union $\bigcup_{r>0}B_r$ is equal to $M'$.
 Hence, using the Arzel\`a-Ascoli theorem (see \cite{alt16})
 there exists a continuous function $f$ on $M'$
 such that a subsequence of $f_{\nu}$ converges to
 $f$ in $C^0_{\loc}$.
 Furthermore,
 in view of the above estimate
 and the assumptions on the covariant derivatives
 of $\alpha$ we obtain
 \[
 \sup_{\nu\in\N}\|f_{\nu}\|_{C^k(B_r)}\leq
 \max\big\{2rC_0, 2C_{k-1}\big\}
 \]
 for all $k\in\N$ and for all $r>0$.
 Therefore, as long as the closure $\bar{B}_r$
 of $B_r$ is contained in a chart domain
 a diagonal sequence argument
 and \cite[Theorem 8.6]{alt16} imply
 that there exists a subsequence $f_{\nu_r}$
 that converges in $C^{\infty}(B_r)$ to the
 {\it a posteriori} smooth function $f|_{B_r}$.
 If $\bar{B}_r$ is not contained in a chart domain
 we can work with a finite covering of $\bar{B}_r$
 by chart domains taking subsequences successively
 with respect to an ordering of the covering chart domains.
 Hence, for all $r\in\N$ there exists a subsequence $f_{\nu_r}$
 that converges in $C^{\infty}(B_r)$ to $f|_{B_r}$.
 A further diagonal sequence argument yields a
 in $C^{\infty}_{\loc}$ converging subsequence $f_{\mu}\ra f$
 so that $\alpha_{\mu}\ra\alpha_0$ in $C^{\infty}_{\loc}$.
 
 In fact, the above argument works
 without making any assumption
 on the first de Rham cohomology
 of the covering space as follows:
 On each open ball $B\subset M'$
 whose closure $\bar{B}$ is contained
 in a chart domain a unique primitive
 function $f_{\nu}$ on $B$ of $(\alpha_{\nu}-\alpha)|_B$
 that vanishes on the centre of $B$
 can be selected.
 Hence, $\alpha_{\mu}|_B\ra \alpha_B$ in $C^{\infty}(B)$
 as $\mu\ra\infty$ for a subsequence of $\alpha_{\nu}$
 using the above argument.
 Taking finite coverings of the closure of $B_r$,
 $r\in\N$, by ball-like chart domains $B$
 and using a diagonal sequence of $\alpha_{\nu}$
 with respect to $B^{11},\ldots,B^{k_11}$ (covering $B_1$),
 $B^{12},\ldots,B^{k_22}$ (covering $B_2$),
 and so on,
 we find a subsequence $\alpha_{\mu}$ of $\alpha_{\nu}$
 that converges in $C^{\infty}_{\loc}$ to $\alpha_0$,
 where for all $r\in\N$ and for all $j=1,\ldots,k_r$
 the restriction of $\alpha_0$ to $B^{jr}$ equals $\alpha_{B^{jr}}$.
 This proves the proposition in general.
\end{proof}

We remark that a global $C^0$-bound on $\alpha$
is not sufficient in order to find a convergent subsequence
of $\alpha_{\nu}$ in $C^0_{\loc}$.
The above proof shows in fact:

\begin{cor}
 \label{cor:concludwithckonly}
 If there exists $k\in\N$ such that $\|\alpha\|_{C^k}<C$
 for a positive constant $C$, then the sequence $\alpha_{\nu}$
 admits a subsequence $\alpha_{\mu}\ra\alpha_0$
 that converges in $C^{k-1}_{\loc}$
 to a $1$-form $\alpha_0$ of class $C^{k-1}$.
\end{cor}


\subsection{Induced convergence on complex structures\label{subsec:indconvcpxstr}}

We consider a covering $\pi\co M'\ra M$
as in Section \ref{subsec:acovering}.
Additionally,
we assume that the primitive $\alpha$ of $\omega'$
is a contact form on $M'$ so that all
$\alpha_{\nu}$ are contact forms.
We denote the induced contact structures
by $\xi_{\nu}:=\ker\alpha_{\nu}$,
which are provided with the symplectic form
obtained by the restrictions of
$\omega'=\rmd\alpha_{\nu}$ to $\xi_{\nu}$.
As shown in Section \ref{subsec:theindcomstr},
for all $\nu\in\N$
there exists a unique section $\Phi_{\nu}$
in the endomorphism bundle of $\xi_{\nu}$
such that
\[
\omega'=g'\big(\Phi_{\nu}(\,.\,),\,.\,\big)
\qquad\text{on}\quad\xi_{\nu}
\;.
\]
We obtain complex structures $j_{\nu}$
on $(\xi_{\nu},\omega')$ by setting
\[
j_{\nu}:=\Phi_{\nu}\circ\big(\sqrt{-\Phi_{\nu}^2}\big)^{-1}
\]
so that
\[
g_{j_{\nu}}:=\omega'(\,.\,,j_{\nu}\,.\,)
\qquad\text{on}\quad\xi_{\nu}
\]
defines a bundle metric on $\xi_{\nu}$.

On the product $\R\times M'$ we consider
the non-degenerate $2$-forms
\[
\eta_{\nu}:=\rmd t\wedge\alpha_{\nu}+\omega'
\]
denoting the $\R$-coordinate by $t$.
Notice that the exterior differentials are equal to
\[
\rmd\eta_{\nu}=-\rmd t\wedge\omega'
\;.
\]
For all $\nu\in\N$ we define a unique endomorphism field
$\Psi_{\nu}$ on $\R\times M'$ by requiring
$\Psi_{\nu}$ to be $\R$-translation invariant
such that the restriction of $\Psi_{\nu}$ to $\xi_{\nu}$
equals $\Phi_{\nu}$ and such that
\[
\Psi_{\nu}(\partial_t)=R_{\nu}
\quad\text{and}\quad
\Psi_{\nu}(R_{\nu})=-\partial_t
\;,
\]
where $R_{\nu}$ denotes the Reeb vector field
of $\alpha_{\nu}$.
Taking the splitting
\[
T\big(\R\times M'\big)=
\R\partial_t\oplus\R R_{\nu}\oplus\xi_{\nu}
\]
into account we get
\[
\Psi_{\nu}=i_{\nu}\oplus\Phi_{\nu}
\;.
\]
Therefore,
\[
J_{\nu}:=
\Psi_{\nu}\circ\big(\sqrt{-\Psi_{\nu}^2}\big)^{-1}=
i_{\nu}\oplus j_{\nu}
\]
is the unique almost complex structure on $\R\times M'$
that is $\R$-translation invariant, restricts to $j_{\nu}$ on
$\xi_{\nu}$, and sends $\partial_t$ to $R_{\nu}$.
Because of
\[
\alpha_{\nu}=-\rmd t\circ J_{\nu}
\]
the bilinear form
\[
\eta_{\nu}(\,.\,,J_{\nu}\,.\,)=
\rmd t\otimes\rmd t+\alpha_{\nu}\otimes\alpha_{\nu}+
g_{j_{\nu}}
\]
is a metric on $\R\times M'$.

Similarly, we define a sequence of metrics
\[
g_{\nu}:=
\rmd t\otimes\rmd t+\alpha_{\nu}\otimes\alpha_{\nu}+
g'|_{\xi_{\nu}}
\]
on $\R\times M'$,
where
\[
g'|_{\xi_{\nu}}(v,w)=
g'\Big(v-\alpha_{\nu}(v)R_{\nu}, w-\alpha_{\nu}(w)R_{\nu}\Big)
\]
for tangent vectors $v,w$ parallel to $M'$.
We remark that
\[
\eta_{\nu}=g_{\nu}\big(\Psi_{\nu}(\,.\,),\,.\,\big)
\]
on $\R\times M'$ and that the equation characterizes
$\Psi_{\nu}$ uniquely.

\begin{lem}
\label{lem:convofaimplstheoneofj}
 We assume that $\alpha$
 satisfies condition \eqref{lowerbound}.
 If $\alpha_{\nu}\ra\alpha_0$ converges in $C^{\infty}_{\loc}$,
 then $\alpha_0$ is a contact form and
 the uniquely associated sequence
 of almost complex structures $J_{\nu}\ra J_0$
 converges in $C^{\infty}_{\loc}$
 to the almost complex structure
 $J_0$ that is defined by $\alpha_0$
 via the above construction.
\end{lem}

\begin{proof}
 First of all observe that
 $\rmd\alpha_{\nu}=\omega'$
 converges to $\rmd\alpha_0$
 in $C^{\infty}_{\loc}$.
 Therefore,
 $\rmd\alpha_0=\omega'$ and
 $\eta_{\nu}$ converges to
 $\eta_0:=\rmd t\wedge\alpha_0+\omega'$
 in $C^{\infty}_{\loc}$.
 Using condition \eqref{lowerbound} we see
 that $|\alpha_{\nu}(v)|\geq c|v|_{g'}$
 for all $v\in\ker\omega'$
 because all $\varphi_{\nu}\in G$
 are isometries.
 In other words the limiting $1$-form
 $\alpha_0$ is a contact form.
 
 We want to show that $R_{\nu}$
 converges to the Reeb vector field $R_0$
 of $\alpha_0$.
 For that we observe $R_{\nu}$ is the unique
 $\R$-invariant vector field on $\R\times M'$
 such that $\iota_{R_{\nu}}\eta_{\nu}=-\rmd t$.
 Choosing local coordinates $x^1=t, x^2, \ldots, x^{2n}$
 we get $R^j=-\eta^{1j}$ for the components
 of $R_{\nu}$, where $\eta^{ij}$ is the inverse
 of the matrix of the $\eta_{\nu}$-coefficients.
 Hence, $R_{\nu}\ra R_0$ in $C^{\infty}_{\loc}$.
 
 This implies that $g'|_{\xi_{\nu}}$ converges
 in $C^{\infty}_{\loc}$ to $g'|_{\xi_0}$,
 which is defined with respect to the
 contact structure $\xi_0=\ker\alpha_0$.
 In order to prove this
 we denote the projection of $TM'$
 onto $\xi_{\nu}$ along the Reeb vector field
 $R_{\nu}$ by $\pi_{\xi_{\nu}}$.
 Setting $P_{\nu}:=0\oplus\pi_{\xi_{\nu}}$ on
 $T(\R\times M')$ we obtain in local coordinates
 $\gamma_{ij}=P_i^kP_j^{\ell}(g')_{k\ell}$
 for the metric coefficients of $g'|_{\xi_{\nu}}$
 ignoring the subscript $\nu$ in the notation for the
 coefficients of $P_{\nu}$.
 Because of $\pi_{\xi_{\nu}}(v)=v-\alpha_{\nu}(v)R_{\nu}$,
 which locally reads as $P_i=\partial_i-\alpha_iR$,
 we get $P_{\nu}\ra P_0$ in $C^{\infty}_{\loc}$,
 where $P_0$ is the projection defined with respect to
 $\alpha_0$.
 Therefore,
 $g'|_{\xi_{\nu}}\ra g'|_{\xi_0}$ in $C^{\infty}_{\loc}$
 as claimed.
 
 We claim that $\Psi_{\nu}$ converges to $\Psi_0$
 in $C^{\infty}_{\loc}$,
 where $\Psi_0$ is defined via the preliminary
 construction with respect to $\alpha_0$.
 Observe that $g_{\nu}$ converges to
 $\rmd t\otimes\rmd t+\alpha_0\otimes\alpha_0+g'|_{\xi_0}$
 in $C^{\infty}_{\loc}$.
 Because of $\eta_{\nu}=g_{\nu}\big(\Psi_{\nu}(\,.\,),\,.\,\big)$
 we get $\Psi_i^j=\eta_{i\ell}\,g^{\ell j}$
 for local representations of the $(1,1)$-tensors $\Psi_{\nu}$.
 Here the right hand side is the product
 of the $\eta_{\nu}$-coefficients
 and the inverse of the metric coefficients of $g_{\nu}$.
 As both converge we get
 $\Psi_{\nu}\ra \Psi_0$ in $C^{\infty}_{\loc}$.
 
 It remains to show that $J_{\nu}\ra J_0$
 converges in $C^{\infty}_{\loc}$,
 where the almost complex structure $J_0$
 is uniquely characterized by
 the preliminary construction.
 But this follows because
 \[
 J_{\nu}=
 \Psi_{\nu}\circ\big(\sqrt{-\Psi_{\nu}^2}\big)^{-1}
 \]
 locally can be expanded in a power series
 in $\Psi_{\nu}$ with coefficients being independent
 of $\nu$.
 This proves $C^{\infty}_{\loc}$-convergence of $J_{\nu}$
 and hence the lemma.
\end{proof}

\begin{rem}
\label{rem:remonckbounds}
In order to conclude
that $J_0$ can be constructed via $\alpha_0$
in the above proof,
in which situation we will simply write $(\alpha_0,J_0)$,
one needs convergence in $C^1_{\loc}$
because we have to differentiate the limiting
$1$-form $\alpha_0$.
Apart from that the above proof goes through
in the case of $C^k_{\loc}$-convergence
for all $k\in\N$.
Combined with Corollary \ref{cor:concludwithckonly}
this implies that if $\alpha$ has a global $C^k$-bound
for $k\in\N$ at least $2$,
then a subsequence of $(\alpha_{\nu},J_{\nu})$
can be selected
that converges in $C^{k-1}_{\loc}$ to $(\alpha_0,J_0)$.
It follows that
$\alpha_0$ admits a global $C^{k-1}$-bound
and satisfies \eqref{lowerbound}.
\end{rem}

\begin{rem}
\label{rem:assalmcomplstrwitha}
 The almost complex structure $J$
 introduced in Section \ref{subsec:almostcpxstr}
 equals
 \[
 J=
 \Psi\circ\big(\sqrt{-\Psi^2}\big)^{-1}
 \;,
 \]
 where the endomorphism field $\Psi$ on $\R\times M'$
 is uniquely determined by
 \[
 \eta=\Big(\rmd t\otimes\rmd t+\alpha\otimes\alpha+g'|_{\xi}\Big)
 \big(\Psi(\,.\,),\,.\,\big)
 \]
 setting $\eta=\rmd t\wedge\alpha+\omega'$.
 We consider the sequence of diffeomorphisms
 $F_{\nu}=a_{\nu}\times\varphi_{\nu}$ on $\R\times M'$,
 where $a_{\nu}$ denotes the translation
 by the real number $a_{\nu}$, 
 and claim that
 \[
 J_{\nu}=F_{\nu}^*J
 \;.
 \]
 Indeed,
 $F_{\nu}^*\eta=\eta_{\nu}$ and
 \[
 F_{\nu}^*\big(\rmd t\otimes\rmd t+\alpha\otimes\alpha+g'|_{\xi}\big)
 =g_{\nu}
 \]
 because of the following observations:
 Recall that $\pi_{\xi}(v)=v-\alpha(v)R$
 and that $g'|_{\xi}=(\pi_{\xi})^*g'$
 so that because of
 $\pi_{\xi}\circ T\varphi_{\nu}=T\varphi_{\nu}\circ\pi_{\xi_{\nu}}$
 we get
 \[
 \varphi_{\nu}^*(g'|_{\xi})=
 \varphi_{\nu}^*\pi_{\xi}^*g'=
 \pi_{\xi_{\nu}}^*\varphi_{\nu}^*g'=
 \pi_{\xi_{\nu}}^*g'=
 g'|_{\xi_{\nu}}
 \;.
 \]
 Therefore,
 \[
 \eta_{\nu}=g_{\nu}
 \big(F_{\nu}^*\Psi(\,.\,),\,.\,\big)
 \;,
 \]
 which characterizes $\Psi_{\nu}$ uniquely.
 In other words $\Psi_{\nu}=F_{\nu}^*\Psi$.
 This implies the claim
 because the eigenvalues of a matrix
 are invariant under conjugations.
\end{rem}


\section{Compactness\label{sec:comp}}

We consider a virtually contact structure
$\big(\pi\co M'\ra M, \alpha, \omega, g\big)$
together with the associated almost complex structure
$J$ on $\R\times M'$ constructed in
Sections \ref{subsec:theindcomstr} and \ref{subsec:almostcpxstr}.
We assume that the covering $\pi$ is regular,
i.e.\ that the group of deck transformations $G$
acts transitively on the fibres of $\pi$.
Furthermore,
we assume that any sequence
$\alpha_{\nu}=\varphi_{\nu}^*\alpha$,
$\varphi_{\nu}\in G$, has a in $C^{\infty}_{\loc}$
converging subsequence.
In view of Lemma \ref{lem:convofaimplstheoneofj}
the associated subsequence of $J_{\nu}$ converges
in $C^{\infty}_{\loc}$ as well.
In the same manner we assume
that for any accumulation point $\alpha_0$ of $\alpha_{\nu}$
the sequence
$\varphi_{\nu}^*\alpha_0$
and the associated sequence
of almost complex structures
have converging subsequences,
cf.\ Remark \ref{rem:remonckbounds}.

Let
 \[
 u_{\nu}=(a_{\nu},f_{\nu})\co
 (\D,\partial\D)\lra
 (\R\times M',\{0\}\times L)
 \]
be a sequence of $J$-holomorphic discs
with boundary in an open relatively compact
subset $K_L$ of a maximally $J$-totally real
submanifold $L\subset M'$.
We assume that the Hofer energy
 \[
 E_{\hh}(u):=\sup_{\tau}\int_{\D}u^*\rmd(\tau\alpha)
 \]
is uniformly bounded by $E>0$ for all $u=u_{\nu}$,
where the supremum is taken over all
smooth increasing functions
$\tau\co\R\ra[0,1]$.
By the maximum principle,
the symplectic energy
 \[
 E(u)=\int_{\D}u^*\rmd(\rme^t\alpha)
 \]
is uniformly bounded by $E$ for all $u=u_{\nu}$ too.

In this section we will carry out
a bubbling off analysis which is largely analogous
to \cite[Section 6]{gz10} and \cite[p.~543-548]{gz13}.
But it is necessary to adapt the arguments
to the present situation of a non-compact contact manifold $M'$.


\subsection{Bubbling off analysis\label{subsec:bubboffanal}}

If the maximum of the $|a_{\nu}|$ over $\D$
is uniformly bounded,
then by Corollary \ref{cor:lowerboundgromconv}
the sequence $u_{\nu}$ has a Gromov convergent subsequence 
that converges to a stable holomorphic disc
whose underlying bubble tree consists of discs only.

If the sequence of maxima is not bounded,
then we find a sequence $\zeta_{\nu}$ in $\D$
such that a subsequence of
$a_{\nu}(\zeta_{\nu})$ tends to $-\infty$.
By the mean value theorem we find a point
$z_{\nu}$ on the line segment connecting
$\zeta_{\nu}$ with $1$ in $\D$
such that
\[
a_{\nu}(\zeta_{\nu})=
T_{z_{\nu}}a_{\nu}\cdot (\zeta_{\nu}-1)
\]
using $a_{\nu}(1)=0$.
Hence,
\[
|a_{\nu}(\zeta_{\nu})|\leq
2\;\big|T_{z_{\nu}}u_{\nu}\big|_{g_0'}
\]
with respect to the complete metric
$g_0'=\rmd t\otimes\rmd t+g'$,
so that the sequence
$|Tu_{\nu}|_{g_0'}$
is not uniformly bounded.
Passing to a further subsequence
and writing
\[
|\nabla u_{\nu}|_{g_0'}=
\sqrt{|\partial_xu_{\nu}|_{g_0'}^2+|\partial_yu_{\nu}|_{g_0'}^2}
\]
instead of $|Tu_{\nu}|_{g_0'}$
we can assume that
\[
R_{\nu}:=
\max_{\D}|\nabla u_{\nu}|_{g_0'}=
|\nabla u_{\nu}(z_{\nu})|_{g_0'}
\lra\infty
\]
for a sequence $z_{\nu}\ra z_0$ in $\D$.
In the following we will distinguish the cases
whether the limit point $z_0$
lies in the interior $B_1(0)\subset\C$
or on the boundary $\partial\D$ of $\D$.

{\bf Case 1:}
$z_0\in B_1(0)$.
We can assume that no $z_{\nu}$
lies on the boundary $\partial\D$.
We choose $\varepsilon>0$
such that $B_{\varepsilon}(z_{\nu})\subset B_1(0)$
for all $\nu\in\N$.
Let $\DD\subset M'$ be a fundamental domain of the covering $\pi$
that contains the base point $o\in M'$
and choose a sequence $\varphi_{\nu}\in G$ of
deck transformations such that
$\varphi_{\nu}^{-1}$ maps $f_{\nu}(z_{\nu})$
into the closure of $\DD$,
see \cite[p.~201]{cha06}.
We consider the rescaled sequence
$v_{\nu}=(b_{\nu},h_{\nu})$ defined via
\[
b_{\nu}(z):=a_{\nu}\big(z_{\nu}+z/R_{\nu}\big)-a_{\nu}(z_{\nu})
\]
and
\[
h_{\nu}(z):=\varphi_{\nu}^{-1}\Big(f_{\nu}\big(z_{\nu}+z/R_{\nu}\big)\Big)
\]
for all $z\in B_{R_{\nu}\varepsilon}(0)$
so that $v_{\nu}(0)\in\{0\}\times\bar{\DD}$.
Moreover, because $F_{\nu}$ is an isometry with respect to
the metric $g_0'$ we have that
$|\nabla v_{\nu}(0)|_{g_0'}=1$
and that
$|\nabla v_{\nu}|_{g_0'}\leq1$
uniformly on $B_{R_{\nu}\varepsilon}(0)$.
Observe that $v_{\nu}$ is obtained from $u_{\nu}$
by a reparametrisation with a M\"obius transformation
and the composition with the inverse of the
diffeomorphism $F_{\nu}=a_{\nu}(z_{\nu})\times\varphi_{\nu}$.
Therefore,
$v_{\nu}$ is $J_{\nu}$-holomorphic
with respect to the almost complex structure
$J_{\nu}=F_{\nu}^*J$ on $\R\times M'$
(see Remark \ref{rem:assalmcomplstrwitha}),
which is associated to $\alpha_{\nu}=\varphi_{\nu}^*\alpha$,
cf.\ Section \ref{subsec:indconvcpxstr}.
We can assume that a further subsequence of $u_{\nu}$
is selected for which the sequence $(\alpha_{\nu},J_{\nu})$
converges in $C^{\infty}_{\loc}$ to $(\alpha_0,J_0)$.
We finally remark that the Hofer energy
\[
\sup_{\tau}\,
\int_{B_{R_{\nu}\varepsilon}(0)}
v_{\nu}^*\,\rmd(\tau\alpha_{\nu})\leq E
\]
is uniformly bounded,
where the supremum is taken over all
smooth increasing functions
$\tau\co\R\ra[0,1]$.

Let $k$ be a natural number
and choose $\nu_0\in\N$ such that
the closure of $B_k(0)$
is contained in
$B_{R_{\nu}\varepsilon}(0)$
for all $\nu\geq\nu_0$.
For all $z\in B_k(0)$ the $g_0'$-distance between
$v_{\nu}(0)$ and $v_{\nu}(z)$
is bounded by $k$ as the uniform gradient bound
$|\nabla v_{\nu}|_{g_0'}\leq1$ shows.
By \cite[p.~117/18 and p.~201]{cha06}
the fundamental domain $\DD$
can be chosen such that $\pi(\DD)$
equals the complement of the cut locus
of $\pi(o)$ in $(M,g)$,
where $o$ denotes the base point of $M'$.
Because the diameter $d_0$ of $(M,g)$ is finite
(as $M$ is compact)
the $g'$-distance between $o$ and $h_{\nu}(0)$
is bounded by $d_0$.
Therefore,
for all $\nu\geq\nu_0$ the discs
\[
v_{\nu}\big(\overline{B_k(0)}\big)
\subset
[-k,0]\times
\overline{B_{d_0+k}(o)}
\]
are contained in the product of the interval
$[-k,0]$ and the closure of the 
$g'$-geodesic ball $B_{d_0+k}(o)$,
which by the Hopf--Rinow theorem is compact.

Because the restriction of
$(\alpha_{\nu},J_{\nu})$
to $[-k,0]\times\overline{B_{d_0+k}(o)}$
converges uniformly with all covariant derivatives
to the restriction of $(\alpha_0,J_0)$,
with elliptic regularity we see that
$v_{\nu}|_{\overline{B_k(0)}}$
has a subsequence $v_{\nu_k}$
that converges uniformly with all derivatives,
see \cite[Theorem B.4.2]{mdsa04} and \cite[p.~559]{gz10}.
Inductively, using a diagonal sequence argument
we see that there exists a converging subsequence
(again denoted by)
\[
v_{\nu}\lra v
\qquad\text{in}\quad
C^{\infty}_{\loc}(\C)
\]
that converges to a non-constant
$J_0$-holomorphic map $v\co\C\ra\R\times M'$.
By Fatou's lemma we find for all $k\in\N$
and all smooth increasing functions $\tau\co\R\ra[0,1]$
\[
\int_{B_k(0)}v^*\,\rmd(\tau\alpha_0)=
\int_{B_k(0)}
\lim_{\nu\ra\infty}
v_{\nu}^*\,\rmd(\tau\alpha_{\nu})\leq 
\liminf_{\nu \to \infty} \int_{B_k(0)} v_\nu^*\,\rmd(\tau\alpha_{\nu})= E
\]
so that $v$ is a finite energy plane with Hofer energy
\[
\sup_{\tau}
\int_{\C}v^*\,\rmd(\tau\alpha_0)\leq E
\]
bounded by $E$.
This finishes our considerations for $z_0\in B_1(0)$.

{\bf Case 2:}
$z_0\in\partial\D$.
We identify $\D\setminus\{-z_0\}$
with the closed upper half-plane $\Hp$
conformally such that
$(z_0,0,-z_0)$ corresponds to $(0,\rmi,\infty)$.
Under this identification we regard $u_{\nu}$
as a $J$-holomorphic map
\[
u_{\nu}\co
(\Hp,\R)\lra
\big(\R\times M',\{0\}\times L\big)
\;.
\]
The sequence of the corresponding
bubble points again denoted by $z_{\nu}\ra 0$
can be assumed to be contained in
\[
\D^+:=\D\cap\Hp
\;.
\]
Moreover, we find a positive constant $c$
such that for all $z\in\D^+$ and for all $\nu\in\N$
\[
\frac1cR_{\nu}\leq|\nabla u_{\nu}(z_{\nu})|_{g_0'}
\quad
\text{and}
\quad
|\nabla u_{\nu}(z)|_{g_0'}\leq cR_{\nu}
\;.
\]
By conformal equivalence the Hofer energy stays unchanged
so that
\[
\sup_{\tau}\,
\int_{\D^+}
u_{\nu}^*\,\rmd(\tau\alpha)\leq E
\]
for all $\nu\in\N$,
where the supremum is taken over all
smooth increasing functions
$\tau\co\R\ra[0,1]$.
Passing to a suitable subsequence
we can assume that there exists $\varepsilon>0$
such that for all $\nu\in\N$
\[
B_{\varepsilon}^+(z_{\nu}):=B_{\varepsilon}(z_{\nu})\cap\Hp
\]
is contained in $\D^+$ and
that $R_{\nu}y_{\nu}\ra\varrho\in[0,\infty]$
writing $z_{\nu}=x_{\nu}+\rmi y_{\nu}$.

{\bf Case 2(a):}
$\varrho=\infty$.
Exactly as in Case 1 one defines a rescaled sequence
$v_{\nu}$ of $J_{\nu}$-holomorphic maps on
$B_{R_{\nu}\varepsilon}(0)\cap\{y\geq-R_{\nu}y_{\nu}\}$
which has the property that
$v_{\nu}(0)$ is contained in $\{0\}\times\bar{\DD}$
and that $|\nabla v_{\nu}(0)|_{g_0'}$
is bounded from below by $1/c$ independently of $\nu$.
Moreover, we have that
$|\nabla v_{\nu}|_{g_0'}\leq c$
uniformly on
$B_{R_{\nu}\varepsilon}(0)\cap\{y\geq-R_{\nu}y_{\nu}\}$
and that the Hofer energy with respect to $\alpha_{\nu}$
satisfies
$\sup_{\nu\in\N}E_{\hh}(v_{\nu})\leq E$.
Arguing as in Case 1 one selects a
$C^{\infty}_{\loc}(\C)$-converging subsequence
$v_{\nu}\ra v$ that converges to a
non-constant $J_0$-holomorphic
finite energy plane $v$ whose Hofer energy
is taken with respect to $\alpha_0$.

{\bf Case 2(b):}
$\varrho<\infty$.
This time we define a rescaled sequence $v_{\nu}$
by
\[
v_{\nu}(z):=u_{\nu}\big(x_{\nu}+z/R_{\nu}\big)
\]
for all $z\in B_{R_{\nu}\varepsilon}^+\big(\rmi R_{\nu}y_{\nu}\big)$,
where $z_{\nu}=x_{\nu}+\rmi y_{\nu}$.
Observe that 
$B_{R_{\nu}\varepsilon}^+\big(\rmi R_{\nu}y_{\nu}\big)$
defines an exhausting sequence of
open subsets of $\Hp$.
Furthermore,
we have that
$|\nabla v_{\nu}(\rmi R_{\nu}y_{\nu})|_{g_0'}$
is greater than or equal to $1/c$
independently of $\nu$,
that 
$|\nabla v_{\nu}|_{g_0'}\leq c$
uniformly on
$B_{R_{\nu}\varepsilon}^+\big(\rmi R_{\nu}y_{\nu}\big)$,
and that the Hofer energy with respect to $\alpha$
satisfies
$\sup_{\nu\in\N}E_{\hh}(v_{\nu})\leq E$.
We remark that 
$v_{\nu}(x)$ is contained in $\{0\}\times K_L$
for all $x\in\R$.

For given $k\in\N$ and $\nu$ such that
$B^+_k(0)$ is contained in
$B_{R_{\nu}\varepsilon}^+\big(\rmi R_{\nu}y_{\nu}\big)$
we find that the $g_0'$-distance from $v_{\nu}(z)$
to $\{0\}\times K_L$ is bounded by $ck$
because of the above uniform gradient bound.
Hence, for all $\nu$ sufficiently large
$v_{\nu}\big(B^+_k(0)\big)$ is a subset of the
$ck$-neighbourhood of $\{0\}\times K_L$
with respect to $g_0'$,
whose closure is compact.
Invoking
\cite[Theorem B.4.2]{mdsa04} and \cite[p.~559]{gz10}
as in Case 1 we find a 
$C^{\infty}_{\loc}(\Hp)$-converging subsequence
$v_{\nu}\ra v$, where
\[
v\co(\Hp,\R)\lra\big(\R\times M',\{0\}\times\overline{K}_L\big)
\]
is a boundary condition preserving
non-constant $J$-holomorphic
finite energy half-plane of Hofer energy
$E_{\hh}(v)\leq E$ with respect to $\alpha$.


\subsection{A finite energy cylinder\label{subsec:afinenergcyl}}

Let $v$ be a non-constant
$J_0$-holomorphic finite energy plane
with Hofer energy $E_{\hh}(v)\leq E$ with respect to $\alpha_0$
which can be obtained as in Section \ref{subsec:bubboffanal}
Case 1 or Case 2(a).
Set
\[
T^1:=\R/2\pi\Z
\;.
\]
Using the conformal map $\R\times T^1\ra\C\setminus\{0\}$
given by $(s,t)\mapsto\rme^{s+\rmi t}$
we regard $v$ as a finite energy cylinder
\[
v\co\R\times T^1\lra\R\times M'
\;.
\]
We claim that $|\nabla v|_{g_0'}$
is bounded globally on $\R\times T^1$.

Arguing by contradiction as in
\cite[Proposition 30]{hof93}
we select sequences
$\varepsilon_{\nu}\ra0$ in $(0,\infty)$,
$z_{\nu}$ on the universal cover $\R\times\R\equiv\C$
with $|z_{\nu}|\ra\infty$, and
$R_{\nu}:=|\nabla v(z_{\nu})|_{g_0'}$,
such that for a subsequence we have
$R_{\nu}\varepsilon_{\nu}\ra\infty$ and
$|\nabla v|_{g_0'}\leq2R_{\nu}$ on
$B_{\varepsilon_{\nu}}(z_{\nu})$ according to
the Hofer lemma \cite[Lemma 26]{hof93}.
Writing $v=(b,h)$
let $\varphi_{\nu}\in G$ be a sequence of
deck transformations whose inverse send
$h(z_{\nu})$ into the closure of the fundamental domain $\DD$.
We define
$u_{\nu}=(a_{\nu},f_{\nu})$ by
\[
a_{\nu}(z):=b\big(z_{\nu}+z/R_{\nu}\big)-b(z_{\nu})
\]
and
\[
f_{\nu}(z):=\varphi_{\nu}^{-1}\Big(h\big(z_{\nu}+z/R_{\nu}\big)\Big)
\]
for all $z\in B_{R_{\nu}\varepsilon_{\nu}}(0)$.
As observed in Section \ref{subsec:bubboffanal}
Case 1 we have that
$u_{\nu}(0)\in\{0\}\times\bar{\DD}$,
$|\nabla u_{\nu}(0)|_{g_0'}=1$,
and that
$|\nabla u_{\nu}|_{g_0'}\leq2$
uniformly on $B_{R_{\nu}\varepsilon_{\nu}}(0)$.
Moreover,
$u_{\nu}$ is $J_0^{\nu}$-holomorphic
with respect to the almost complex structure
$J_0^{\nu}=F_{\nu}^*J_0$ on $\R\times M'$,
where
$F_{\nu}=b(z_{\nu})\times\varphi_{\nu}$,
and the Hofer energy
with respect to
$\alpha_0^{\nu}=\varphi_{\nu}^*\alpha_0$
\[
\sup_{\tau}\,
\int_{B_{R_{\nu}\varepsilon_{\nu}}(0)}
u_{\nu}^*\,\rmd(\tau\alpha_0^{\nu})\leq E
\]
is uniformly bounded,
where the supremum is taken over all
smooth increasing functions
$\tau\co\R\ra[0,1]$.
Similarly,
\[
\int_{B_{R_{\nu}\varepsilon_{\nu}}(0)}
f_{\nu}^*\,\rmd\alpha_0^{\nu}=
\int_{B_{\varepsilon_{\nu}}(0)}
h^*\,\rmd\alpha_0
\lra 0
\;.
\]
Additionally, we can assume that a subsequence
of $(\alpha_0^{\nu},J_0^{\nu})$
converges in $C^{\infty}_{\loc}$
to $(\alpha_{\infty},J_{\infty})$.
Using the argumentation from
Section \ref{subsec:bubboffanal} Case 1 
we find a $C^{\infty}_{\loc}(\C)$-converging
subsequence $u_{\nu}\ra u$,
where $u$ is a non-constant
$J_{\infty}$-holomorphic finite energy plane
of Hofer energy
\[
\sup_{\tau}
\int_{\C}u^*\,\rmd(\tau\alpha_{\infty})\leq E
\;.
\]
For the {\bf contact area} we obtain
\[
\int_{\C}f^*\,\rmd\alpha_{\infty}=0
\]
writing $u=(a,f)$ because for all $k\in\N$
\[
\int_{B_k(0)}f^*\,\rmd\alpha_{\infty}=
\int_{B_k(0)}
\lim_{\nu\ra\infty}
f_{\nu}^*\,\rmd\alpha_0^{\nu}\leq
\liminf_{\nu\ra\infty}
\int_{B_{R_{\nu}\varepsilon_{\nu}}(0)}
f_{\nu}^*\,\rmd\alpha_0^{\nu}=0
\]
by Fatou's lemma.
Using \cite[Lemma 28]{hof93},
which holds for non-compact
contact manifolds $(M',\alpha_{\infty})$ as well,
this is a contradiction to $u$ being non-constant.
Therefore, the gradient $|\nabla v|_{g_0'}$
of the finite energy cylinder $v=(b,h)$ is globally bounded.

Modifying the proof of \cite[Theorem 31]{hof93}
we choose a sequence $\varphi_{\nu}\in G$
such that $\varphi_{\nu}^{-1}$ maps $h(\nu,0)$
into the closure of $\DD$.
We define $u_{\nu}=(a_{\nu},f_{\nu})$ via
\[
a_{\nu}(s,t):=
b(s+\nu,t)-b(\nu,0)
\]
and
\[
f_{\nu}(s,t):=\varphi_{\nu}^{-1}\big(h(s+\nu,t)\big)
\;.
\]
Observe that $u_{\nu}(0,0)\in\{0\}\times\bar{\DD}$
and that 
$u_{\nu}$ is $J_0^{\nu}$-holomorphic
with respect to the almost complex structure
$J_0^{\nu}=F_{\nu}^*J_0$ on $\R\times M'$ for
$F_{\nu}=b(\nu,0)\times\varphi_{\nu}$.
The gradient $|\nabla u_{\nu}|_{g_0'}$
is globally bounded on $\R\times T^1$
because $F_{\nu}$ is an isometry with respect to $g_0'$.
Because $J_0^{\nu}$ is associated to
$\alpha_0^{\nu}=\varphi_{\nu}^*\alpha_0$
we remark that the Hofer energy satisfies
\[
\sup_{\tau}\,
\int_{\R\times T^1}
u_{\nu}^*\,\rmd(\tau\alpha_0^{\nu})\leq E
\]
for all $\nu\in\N$.
Moreover, for given $k\in\N$
\[
\int_{[-k,k]\times T^1}
f_{\nu}^*\,\rmd\alpha_0^{\nu}=
\int_{[-k+\nu,k+\nu]\times T^1}
h^*\,\rmd\alpha_0
\lra0
\]
as $\nu$ tends to $\infty$
because the contact area $h^*\,\rmd\alpha_0$
is non-negative (see Section \ref{subsec:almostcpxstr})
and the total integral
\[
\int_{\C}
h^*\,\rmd\alpha_0
\leq E
\]
is bounded by the Hofer energy.
Similarly,
\[
\int_{\{0\}\times T^1}
f_{\nu}^*\,\alpha_0^{\nu}=
\int_{\{\nu\}\times T^1}
h^*\,\alpha_0=
\int_{(-\infty,\nu)\times T^1}
h^*\,\rmd\alpha_0
\]
tends to the contact area
$\int_{\C}h^*\,\rmd\alpha_0$
as $\nu$ tends to $\infty$.
Assuming that a subsequence
of $(\alpha_0^{\nu},J_0^{\nu})$
converges to $(\alpha_{\infty},J_{\infty})$
in $C^{\infty}_{\loc}$
we find as in
Section \ref{subsec:bubboffanal} Case 1 
a $C^{\infty}_{\loc}(\C)$-converging
subsequence $u_{\nu}\ra u$ that converges
to a non-constant $J_{\infty}$-holomorphic
finite energy cylinder $u=(a,f)\co\R\times T^1\ra\R\times M'$
of Hofer energy
\[
\sup_{\tau}
\int_{\R\times T^1}u^*\,\rmd(\tau\alpha_{\infty})\leq E
\;,
\]
vanishing contact area
\[
\int_{\R\times T^1}
f^*\,\rmd\alpha_{\infty}=0
\;,
\]
and action
\[
\int_{\{0\}\times T^1}
f^*\,\alpha_{\infty}=
\int_{\C}h^*\,\rmd\alpha_0
\;,
\]
which by \cite[Lemma 28]{hof93}
does not vanish.
Further, the gradient $|\nabla u|_{g_0'}$
is globally bounded on $\R\times T^1$.
As in the last part of the proof of
\cite[Theorem 31]{hof93}
one shows that a subsequence of
$f_{\nu}(0,\,.\,)$ converges in $C^{\infty}(T^1)$
to a contractible
periodic $\alpha_{\infty}$-Reeb orbit
whose $\alpha_{\infty}$-action
is equal to the contact area
$\int_{\C}h^*\,\rmd\alpha_0$ of $v=(b,h)$.


\subsection{Removal of boundary singularity\label{subsec:remofboundsing}}

In Section \ref{subsec:bubboffanal} Case 2(b)
we obtained a non-constant $J$-holomorphic
finite energy half-plane $v$ of Hofer energy
$E_{\hh}(v)\leq E$ with respect to $\alpha$.
Let $S$ be the strip $\R\times[0,\pi]$,
which is conformally equivalent to $\Hp\setminus\{0\}$
via $(s,t)\mapsto\rme^{s+\rmi t}$.
Using this identification
we regard $v$ as finite energy strip
\[
v\co(S,\partial S)\lra\big(\R\times M',\{0\}\times\overline{K}_L\big)
\;.
\]
We claim that $|\nabla v|_{g_0'}$
is bounded globally on $S$.

Arguing by contradiction  as in
\cite[Theorem 32]{hof93}
we obtain with help
of the Hofer lemma \cite[Lemma 26]{hof93}
$\varepsilon_{\nu}\ra0$ in $(0,\infty)$,
$z_{\nu}=x_{\nu}+\rmi y_{\nu}$ on the strip $S$
with $|x_{\nu}|\ra\infty$, and
$R_{\nu}:=|\nabla v(z_{\nu})|_{g_0'}$,
such that a subsequence satisfies
$R_{\nu}\varepsilon_{\nu}\ra\infty$ and
$|\nabla v|_{g_0'}\leq2R_{\nu}$ on
$B_{\varepsilon_{\nu}}(z_{\nu})$.
Furthermore, we can assume that
$R_{\nu}y_{\nu}\ra\varrho_0\in[0,\infty]$
and
$R_{\nu}(\pi-y_{\nu})\ra\varrho_{\pi}\in[0,\infty]$.
If both $\varrho_0$ and $\varrho_{\pi}$
equal $\infty$ we can argue
as in the first half of Section \ref{subsec:afinenergcyl}
to show that a rescaled sequence $u_{\nu}$
converges in $C^{\infty}_{\loc}(\C)$ to a non-constant
$J_{\infty}$-holomorphic finite energy plane
with respect to $\alpha_{\infty}$ that has vanishing
contact area.
By \cite[Lemma 28]{hof93}
this is a contradiction.
If at least one of the limits $\varrho_0$ or $\varrho_{\pi}$
is finite, where in the second case we precompose $v$
with $z\mapsto-(z-\rmi\pi)$,
we repeat the argumentation from
Section \ref{subsec:bubboffanal} Case 2(b),
i.e.\ the rescaled sequence $u_{\nu}$
converges in $C^{\infty}_{\loc}(\Hp)$ to a non-constant
$J_{\infty}$-holomorphic finite energy half-plane
with respect to $\alpha_{\infty}$ with vanishing
contact area $\int_{\Hp}u^*\rmd\alpha_{\infty}=0$.
This is a contradiction as shown in
the first half of the proof of \cite[Theorem 32]{hof93}.
Hence, the gradient $|\nabla v|_{g_0'}$
of the finite energy strip $v$ is globally bounded.

Finally, with the second
half of the proof of \cite[Theorem 32]{hof93}
(cf.\ \cite[Lemma 6.2]{gz13})
we see that $v$ extends to a holomorphic disc
\[
v\co(\D,\partial\D)\lra\big(\R\times M',\{0\}\times\overline{K}_L\big)
\]
after a suitable conformal change of the parametrization.


\subsection{Aperiodicity and Gromov convergence\label{subsec:aperiodandgromconv}}

We summarize the results of the proceeding section.
Let
$\big(\pi\co M'\ra M, \alpha, \omega, g\big)$
be a virtually contact structure
over a closed connected manifold $M$.
Denote by $J$ the associated almost complex structure
on $\R\times M'$ that is given by the construction in
Section \ref{subsec:theindcomstr} and Section \ref{subsec:almostcpxstr}.
Let $L\subset M'$ be a submanifold such that $\{0\}\times L$
maximally $J$-totally real in $\R\times M'$.
Let
 \[
 u_{\nu}=(a_{\nu},f_{\nu})\co
 (\D,\partial\D)\lra
 (\R\times M',\{0\}\times L)
 \]
be a sequence of $J$-holomorphic discs
with boundary in an open relatively compact
subset $K_L$ of $L$ such that the Hofer energy
 \[
 E_{\hh}(u_{\nu}):=\sup_{\tau}\int_{\D}u_{\nu}^*\rmd(\tau\alpha)
 \]
is uniformly bounded by $E>0$ for all ${\nu}\in\N$,
where the supremum is taken over all
smooth increasing functions
$\tau\co\R\ra[0,1]$.

A {\bf closed characteristic} on $(M,\omega)$
is a compact leaf of the $1$-dimensional
foliation $\ker\omega$.
Taking any orientation on a
closed characteristic it
represents a homology class.
If a non-trivial multiple of a closed characteristic
admits a contractible parametrization
by a loop we will say that the
closed characteristic is {\bf contractible}.

\begin{prop}
 \label{prop:aperioidgromconsum}
 In the situation described above
 we assume that the covering $\pi$ is regular
 and that the $C^3$-norm of $\alpha$ defined in
 Section \ref{subsec:hordcovderiv} is finite.
 If $(M,\omega)$ has no contractible
 closed characteristic,
 then $u_{\nu}$ has a Gromov converging
 subsequence that converges to a stable
 $J$-holomorphic disc with boundary on
 $\overline{K}_L\subset L$ whose underlying
 bubble tree consists of discs only.
\end{prop}

\begin{proof}
 First we will prove the proposition under the additional
 assumption that all $C^k$-norms of $\alpha$ are bounded.
 By Proposition \ref{prop:arzelaascoli}
 this implies that for any sequence $\varphi_{\nu}\in G$
 a in $C^{\infty}_{\loc}$ converging subsequence of
 $\alpha_{\nu}=\varphi_{\nu}^*\alpha$ can be selected.
 The associated subsequence of almost complex structures
 $J_{\nu}$ on $\R\times M'$ converges in $C^{\infty}_{\loc}$ too,
 see Lemma \ref{lem:convofaimplstheoneofj}.
 The analogues convergence statement with $(\alpha,J)$
 replaced by a resulting limit $(\alpha_0,J_0)$
 of a subsequence of $(\alpha_{\nu},J_{\nu})$ holds true as well,
 see Remark \ref{rem:remonckbounds}.
 Observe that any periodic $\alpha_{\infty}$-Reeb orbit,
 which could be obtained as in Section \ref{subsec:afinenergcyl},
 is a closed characteristic of $\rmd\alpha_{\infty}=\omega'$
 and, hence, projects to a closed characteristic of $\omega$ via $\pi$.
 Therefore, in the bubbling off analysis in Section \ref{subsec:bubboffanal}
 only Case 2(b) can occur.
 By Section \ref{subsec:remofboundsing}
 any resulting finite energy half-plane
 extends to a $J$-holomorphic disc
 and has Hofer energy less than or equal to $E$,
 which in the case of a disc equals the contact area.
 In fact, the contact areas of bubbled finite energy discs
 are uniformly bounded from below
 as a bubbling off argument as in
 \cite[Lemma 35]{hof93}
 and modified as in Sections
 \ref{subsec:bubboffanal},
 \ref{subsec:afinenergcyl}, and
 \ref{subsec:remofboundsing}
 shows.
 
 Therefore,
 we can prove the proposition
 with the following arguments:
 As remarked at the beginning of
 Section \ref{subsec:bubboffanal}
 we have to rule out that there exists a sequence
 $\zeta_{\nu}$ in $\D$ such that the minimum of
 $a_{\nu}$ is attained for all $\nu\in\N$
 and $a_{\nu}(\zeta_{\nu})\ra-\infty$ as $\nu\ra\infty$
 for a subsequence.
 By the maximum principle
 we can assume that $\zeta_{\nu}$
 is contained in the interior $B_1(0)$
 of the unit disc $\D$.
 Because we are free to precompose $u_{\nu}$
 by a M\"obius transformation
 we can assume that $\zeta_{\nu}=0$ for all $\nu$.
 By the preliminary remarks of the proof
 we have that all accumulation points $z_0$ of
 sequences $z_{\nu}$ of bubbling points of $u_{\nu}$
 are contained in the boundary $\partial\D$.
 In fact, there are only finitely many of them.
 Indeed, this is because the contact area
 $\int_{\D}f_{\nu}^*\rmd\alpha$ of $u_{\nu}$
 is uniformly bounded by the Hofer energy of $u_{\nu}$
 and, hence, by $E$.
 On the other hand the contact area
 of all possible bubbling discs
 is uniformly bounded from below.
 As the contact area is additive the arguments
 used to prove convergence modulo bubbling
 in \cite[Section 2.5]{fz15}
 carry over to the present context
 as in \cite[p.~542/43]{hof93}.
 In other words, there are boundary points
 $z_0^1,\ldots,z_0^N$ of $\D$
 such that for all $\varepsilon>0$
 the restriction of $|\nabla u_{\nu}|_{g_0'}$ to
 \[\D\setminus\bigcup_{j=1}^NB_{\varepsilon}(z_0^j)\]
 is uniformly bounded.
 Using a mean value argument to obtain $C^0$-bounds
 and \cite[Theorem B.4.2]{mdsa04} we find a subsequence
 of $u_{\nu}$ that converges in $C^{\infty}_{\loc}$
 on the deleted disc $\D\setminus\{z_0^1,\ldots,z_0^N\}$.
 This contradicts $a_{\nu}(0)\ra-\infty$.
 This proves the proposition under the assumption
 that $\alpha$ is bounded with respect to all $C^k$-norms.
 
 In order to obtain the proposition
 under the weaker assumption to have only
 $C^3$-bounds on $\alpha$ we observe that
 we have to select converging subsequences
 of $(\alpha_{\nu},J_{\nu})$ precisely in Case 1
 and Case 2(a) in Section \ref{subsec:bubboffanal}.
 This involves the Arzel\`a--Ascoli argument
 and drops regularity by $1$,
 see Corollary \ref{cor:concludwithckonly}
 and Remark \ref{rem:remonckbounds}.
 In order to obtain global gradient bounds for
 finite energy cylinders in
 Section \ref{subsec:afinenergcyl}
 we have to select converging subsequences
 of $(\alpha_0^{\nu},J_0^{\nu})$ via an
 Arzel\`a--Ascoli argument as discussed
 dropping regularity by $1$ once more.
 The elliptic convergence holds for
 almost complex structures $J_{\nu}$
 that converge in $C^1_{\loc}$
 resulting in $C^1_{\loc}$-converging subsequences
 of $u_{\nu}$,
 see \cite[Remark B.4.3]{mdsa04}.
 Because finite energy planes of class $C^1$
 converge to periodic $C^1$-orbits by the arguments in
 \cite[Theorem 31]{hof93}
 we see that we can work with an {\it a priori}
 $C^3$-bound.
\end{proof}


\section{Contractible closed characteristics\label{sec:contrclosedchcts}}

In this section we will give the main applications
of the compactness results of Section \ref{sec:comp}
in view of our periodicity questions of magnetic flows.
The proofs of Theorem \ref{thmintr:truncclasshammotion}
and \ref{thmintr:hightruncclasshammotion}
are given.


\subsection{Germs of holomorphic discs\label{subsec:germofhlomdiscs}}

We consider a non-trivial
virtually contact structure
$\big(\pi\co M'\ra M, \alpha, \omega, g\big)$
over a closed connected manifold $M$
together with the associated almost complex structure $J$
on $\R\times M'$ that is given by the construction in
Section \ref{subsec:theindcomstr}
and Section \ref{subsec:almostcpxstr}.
The induced contact structure on $M'$ defined by $\alpha$
is denoted by $\xi$.

\begin{thm}
 \label{thm:3dovertwist}
 If the $C^3$-norm of $\alpha$ is finite
 and if $(M',\xi)$ is a $3$-dimensional overtwisted
 contact manifold, then $(M,\omega)$
 admits a contractible closed characteristic.
\end{thm}

\begin{proof}
 We can assume that the covering $\pi$ is regular
 as we always can pass to the universal covering.
 Denote by $D$ an overtwisted disc in $(M',\xi)$
 such that the characteristic foliation $D_{\xi}$
 has a unique singularity and a unique closed leaf
 given by the boundary $\partial D$.
 Denote by $e$ the singularity of $D_{\xi}$.
 Let $U$ be a open ball neighbourhood
 of $e$ whose closure in $M'$ is compact.
 Let $J_U$ be an almost complex structure on $\R\times M'$
 that is translation invariant, sends $\partial_t$ to $R$,
 and restricts to a complex structure
 on $(\xi,\omega')$ such that $J_U$ equals $J$
 in a neighbourhood of $\R\times (M'\setminus U)$
 and allows a local $J_U$-holomorphic
 Bishop disc family emerging from $e$
 in the sense of \cite[Section 4.2]{hof93}
 or \cite[Section 3.1]{hof99}.
 One considers the moduli space
 of all $J_{U}$-holomorphic discs
 with three marked points geometrically fixed
 by three mutually distinct leaves of $D_{\xi}$
 not being $\partial D$ that are homologous
 relative $D\setminus\{e\}$ to one of the Bishop discs,
 see \cite{hof93,hof99}.
 According to automatic transversality,
 positivity of intersections, and
 the relative adjunction inequality
 as worked out in \cite{gz10}
 the evaluation map from the moduli space to one
 of the distinguished leaves
 is a local diffeomorphism
 under which the Bishop discs
 have unique preimages.
 By E.\ Hopf's boundary
 version of the maximum principle
 all holomorphic discs with boundary
 on the totally real punctured disc $D\setminus\{e\}$
 are transverse to $D_{\xi}$,
 see \cite{hof93,hof99}.
 In particular, no holomorphic disc
 can touch the boundary $\partial D$.
 Therefore, the moduli space
 cannot be compact as this would imply
 surjectivity of the evaluation map.
 
 Arguing indirectely
 we assume that $(M,\omega)$
 has no contractible closed characteristic
 so that $(M',\alpha)$ does not have a
 contractible periodic Reeb orbit.
 We claim that under this assumption
 the above moduli space will be compact
 leading to the desired contradiction.
 A uniform Hofer energy bound is given by
 the $\frac12|\rmd\alpha|_{g'}$-area of $D$,
 see \cite[Lemma 33]{hof93}.
 Compactness is a consequence of
 Proposition \ref{prop:aperioidgromconsum}
 with the following modifications:
 In Proposition \ref{prop:distestihofenerg}
 one compares the distance to the set
 $D\cup U$ instead to $L$ so that
 Corollary \ref{cor:lowerboundgromconv}
 still holds.
 In the bubbling off analysis
 leading to finite energy planes
 as in Case 1 and Case 2(a)
 of Section \ref{subsec:bubboffanal}
 and of finite energy cylinders
 as in Section \ref{subsec:afinenergcyl},
 respectively,
 we distinguish the cases whether the $g'$-distance
 of the $f_{\nu}(z_{\nu})$, resp., $h(z_{\nu})$ 
 to $U$ stay bounded.
 If so one argues as in \cite{hof93,hof99}.
 If the distance is unbounded
 one replaces $R_{\nu}\varepsilon$,
 resp., $R_{\nu}\varepsilon_{\nu}$
 by the maximal radius $R_{\nu}'$
 less or equal to $R_{\nu}\varepsilon$,
 resp., $R_{\nu}\varepsilon_{\nu}$
 such that $h_{\nu}$,
 resp., $f_{\nu}$ maps $B_{R_{\nu}'}(0)$
 into $M'\setminus U$.
 With this modification the proof of
 Proposition \ref{prop:aperioidgromconsum}
 goes through (cf.\ \cite[p.~547]{gz13})
 as one can argue with $J$ instead of $J_U$.
\end{proof}

\begin{thm}
 \label{thm:3dpi2nonzero}
 If the $C^3$-norm of $\alpha$ is finite
 and if $M$ is a $3$-dimensional manifold
 with non-trivial $\pi_2M$, then $(M,\omega)$
 admits a contractible closed characteristic.
\end{thm}

\begin{proof}
 Observe that $\pi_2M'$ is non-trivial as well.
 By the sphere theorem there exists
 a non-contractible embedding of a two sphere
 into $M'$ whose image we denote by $S$.
 By Theorem \ref{thm:3dovertwist}
 it suffices to consider the case
 of a tight contact structure $\xi$
 so that we can assume
 that the characteristic foliation
 of $S$ has precisely two singular points
 $e^+$ and $e^-$, which are elliptic,
 but has no limit cycle,
 see \cite[Section 4.6]{gei08}.
 One uses a filling by holomorphic discs
 argument as in the proof of
 Theorem \ref{thm:3dovertwist}
 based on two Bishop disc families
 emerging from $e^+$ and $e^-$,
 respectively.
 A contradiction to the non-existence of
 contractible closed characteristics
 can be obtained as
 compactness of the moduli space
 corresponding to the Bishop families
 would result in a $3$-ball inside $M'$
 that is bounded by $S$,
 see \cite{hof93} and cf.\ \cite{gz10},
 which is not possible.
\end{proof}

An embedded $(2n-2)$-sphere $S$ in $(M',\xi)$
is called {\bf standard} provided that the
restriction of the contact form $\alpha$ to $TS$
equals the restriction of
$\frac12\big(\mathbf{x}\rmd\mathbf{y}-\mathbf{y}\rmd\mathbf{x}\big)$
to $TS$,
where we identify $S$
with the unit sphere $S^{2n-2}$
in $\R\times\R^{n-1}\times\R^{n-1}$
equipped with coordinates $(w,\mathbf{x},\mathbf{y})$,
see \cite[p.~326/27]{gz16a}.

\begin{thm}
 \label{thm:2n-1dpi2n-2nonzero}
 Let $n\geq3$.
 If the $C^3$-norm of $\alpha$ is finite
 and if $(M',\xi)$ is a $(2n-1)$-dimensional
 contact manifold that contains a standard sphere $S$
 whose class $\{S\}$ in $\pi_{2n-2}M'$ is non-trivial,
 then $(M,\omega)$ admits a contractible
 closed characteristic.
\end{thm}

\begin{proof}
 We equip $\R\times S^{2n-2}$
 with the standard contact structure
 defined by the identification
 with the upper boundary of an index $1$ handle,
 see \cite[p.~329]{gz16a}.
 According to \cite[Proposition 6.4]{dige12}
 we find a contact embedding of
 $(-2,2)\times S^{2n-2}$ into $(M',\xi)$
 mapping $\{0\}\times S^{2n-2}$ onto $S$.
 We identify $(-2,2)\times S^{2n-2}$
 with its image $U$ in $M'$.
 Let $\alpha_1$ be the $\xi$-defining
 contact form on $M'$ that coincides with
 $\alpha$ in a neighbourhood of $M'\setminus U$
 and with the $\xi$-defining contact form
 on the upper boundary of the $1$-handle
 in a neighbourhood of $[-1,1]\times S^{2n-2}$
 by convex interpolation.
 Performing a reversed contact surgery
 along the belt sphere $S$
 we see that $(M',\xi)$ is the result
 of an index $1$ surgery
 on a contact manifold $(N,\eta)$.
 The complement of the surgery region in $N$
 equals $M'\setminus U$ and
 $\eta$ admits a defining contact form
 that is given by $\alpha$ on $M'\setminus U$.
 
 This places us in the situation of
 \cite[Chapter 2]{gz16a}.
 We assume that the index $1$ surgery
 described in \cite[Section 2.1]{gz16a}
 is realized in such a way that
 the contact form that corresponds to the thin handle
 equals $\alpha_1$.
 This eventually requires
 a global contactomorphism with support in $U$
 along which $\alpha$ is assumed to be pulled back.
 The contact form $\alpha_R$ corresponding to the thick
 handle can be assumed to coincide with $\alpha$
 in a neighbourhood of $M'\setminus U$.
 The scaling factor $R$ is chosen according to
 \cite[Lemma 6]{gz16a}.
 All this eventually results in a multiplication
 of $\alpha$ by a large constant
 which we are free to ignore in the following.
 
 We consider the manifold $W$
 obtained by gluing the region between the
 thin and the thick handle inside the surgery model
 (see \cite[Section 2.1]{gz16a})
 to $(-\infty,0]\times M'$ along the thin handle
 being contained in $\{0\}\times U$.
 Similarly to \cite[Section 2.2]{gz16a}
 we equip $W$ with the symplectic form
 $\rmd\lambda$, where $\lambda$ on the model region
 is given by the dual of the Liouville vector field on
 \cite[p.~330]{gz16a}.
 On $(-\infty,0]\times M'$ we define $\lambda$
 as follows:
 Because $\alpha$ and $\alpha_1$ define
 the same contact structure
 we find a function $h$ on $M'$
 that has {\it a posteriori} compact support in $U$
 and satisfies $\alpha=\rme^h\alpha_1$.
 The Liouville form $\lambda$
 at $(t,p)$ on $(-\infty,0]\times M'$
 by definition is given by $\rme^{t+b(t)h}\alpha_1$
 for a smooth function $b$ on $\R$
 that equals $0$ on $[0,\infty)$,
 $1$ on $(-\infty,t_0]$ for a suitable $t_0<0$,
 and satisfies $b'(t)h>-1$ on $M'$ for all $t$.
 
 We remark that the interior of the complement
 of $(t_0,0)\times U$ has two
 connected components.
 On the unbounded component
 $\lambda$ is given by $\rme^t\alpha$,
 on which the almost complex structure $J$
 is defined.
 We define a compatible almost complex structure
 $J_U$ on $(W,\rmd\lambda)$
 that extends $J$ such that $J_U$
 equals the complex structure on the model region,
 see \cite[(J2) p.~331]{gz16a},
 is chosen to be generic on $(t_0,0)\times U$
 in the sense of \cite[Section 3.5]{gz16a},
 and makes the boundary of $W$,
 which is equal to a copy of $M'$ with
 $\lambda|_{T\partial W}=\alpha_R$, $J_U$-convex.
 In other words
 we are in a situation as described in
 \cite[Chapter 3]{gz16a} besides the fact
 that the entire construction is based
 on a non-compact contact manifold $(M',\xi)$.
 
 We define a moduli space $\WW$
 of holomorphic discs
 as in \cite[Section 3.2]{gz16a}.
 All properties formulated in \cite[Chapter 3]{gz16a}
 hold true with the following
 modification in the compactness argument:
 Assuming aperiodicity of the $\alpha$-Reeb flow
 on the universal cover $M'$,
 to which we without loss of generality can pass,
 we can prove compactness as in 
 \cite[Section 3.4]{gz16a} for all
 $J_U$-holomorphic discs belonging to $\WW$
 that have uniformly bounded projections to $M'$,
 cf.\ \cite[Chapter 6]{gz13}.
 Otherwise one argues as at the end of the proof of
 Theorem \ref{thm:3dovertwist}.
 In other words,
 $\WW$ is a $(2n-3)$-dimensional
 compact manifold with boundary.

 If $\WW$ is not connected, 
 we replace it by its connected
 component that contains the standard holomorphic discs,
 see \cite[Section 3.1]{gz16a}.
 A deformation of the evaluation map
 as described in \cite[Chapter 4 and Section 5.1]{gz16a}
 yields a continuous map
 \[
 f\co
 \big(\WW\times\D,\partial(\WW\times\D)\big)
 \lra
 (M',S)
 \]
 transverse to the belt sphere $S$
 whose restriction to the boundary of $\WW\times\D$
 has degree $1$
 and which maps $\WW\times\{1\}$
 to a $(2n-3)$-dimensional cell contained in $S$.
 A first consequence is that the homology class
 represented by $S$ in $M'$ must vanish
 so that $S$ separates $M'$.
 Denote the closures of the connected
 components of $M'\setminus S$
 by $M_1$ and $M_2$.
 Notice, that at least one of the $M_i$
 has to be non-compact as $M'$ is not compact
 by our assumption on the non-triviality
 of the virtually contact structure.
 We denote by $V_i$ the preimage
 of $M_i$ under $f$ and observe
 that the restriction $f_i$ of $f$ to $V_i$
 has a well-defined mapping degree.
 Counting the number of preimages with signs
 of generic points near $S$
 of different components of $M'\setminus S$
 we obtain $\deg f_1-\deg f_2=\pm1$
 as in \cite[Lemma 8]{gz16a}.
 Therefore,
 after eventually changing the notation
 we get that $M_1$ is compact
 and the degree of $f_1$ equals $\pm1$.
 By \cite[Proposition 11]{gz16a},
 which uses that $\WW\times\{1\}$
 gets mapped to a cell, we infer that
 $M_1$ is simply connected.
 Moreover, with \cite[Proposition 12]{gz16a}
 $M_1$ has the homology of a ball.
 Because $M_1$ is bounded by a $(2n-2)$-sphere
 for $n\geq3$ the $h$-cobordism theorem implies
 that $M_1$ has to be a ball,
 see \cite[Proposition A on p.~108]{miln65}.
 Hence, the belt sphere $S$ is contractible in $M'$.
 This contradiction shows that the $\alpha$-Reeb flow
 on $M'$ cannot be aperiodic.
\end{proof}

\begin{rem}
 \label{rem:connsumimplsexoforb}
 Let $(M,\omega)$ be a closed connected
 odd-symplectic manifold
 that admits a non-trivial
 virtually contact structure
 $\big(\pi\co M'\ra M,\alpha,\omega,g\big)$.
 Assume that $\big(\pi\co M'\ra M,\alpha,\omega,g\big)$
 is obtained by a covering connected sum
 (as defined in \cite[Section 2.2]{wz16}) 
 up to a multiplication of $\alpha$ by a positive function
 that is different from $1$ only on a compact subset of $M'$.
 We assume that $\alpha$ is $C^3$-bounded,
 cf.\ Remark \ref{rem:remonsomecontconnsum}.
 If the induced connected sum
 of the underlying manifold $M$
 is non-trivial meaning that none of the summands
 is a homotopy sphere,
 then there exists a
 contractible closed characteristic on $(M,\omega)$.
 
 Indeed,
 the case $n=2$ follows with Theorem \ref{thm:3dpi2nonzero}
 denoting the dimension of $M$ by $2n-1$,
 cf.\ \cite[Theorem 1]{gz16a}.
 For the case $n\geq3$ we can assume
 the situation of the proof of Theorem \ref{thm:2n-1dpi2n-2nonzero}.
 Under the assumption of aperiodicity
 one shows that any belt sphere of the
 covering connected sum
 bounds a $(2n-1)$-dimensional disc.
 With the arguments of
 \cite[Proposition 1.6 and Proposition 3.10]{Hatcher}
 the existence of a closed characteristic
 on $(M,\omega)$ follows.
\end{rem}

We say that the {\bf standard contact handle}
of index $k$
admits a contact embedding into 
$(M',\xi)$ if $(M',\xi)$ contains the image of
the upper boundary
$D^k\times S^{2n-1-k}$
of a symplectic handle
$D^k\times D^{2n-k}$
under a suitable contact embedding
as described in \cite[Section 6.2]{gei08}
in the context of contact surgery.
We call $\{\mathbf{0}\}\times S^{2n-1-k}$
or its image the {\bf belt sphere} of the handle.
The index $k$ of the handle is
{\bf subcritical} provided that $k\leq n-1$.

\begin{thm}
 \label{thm:2n-1dhnnonzero}
 Let $n\geq3$.
 If the $C^3$-norm of $\alpha$ is finite
 and if $(M',\xi)$ is a $(2n-1)$-dimensional
 contact manifold that admits an embedding
 of the standard contact handle
 of subcritical index $k$
 with belt sphere $S$
 whose class $[S]$
 in the oriented bordism group
 $\Omega^{SO}_{2n-1-k}M'$
 is non-trivial,
 then $(M,\omega)$ admits a contractible
 closed characteristic.
\end{thm}

\begin{proof}
 In view of \cite[Theorem 1.4]{gnw16}
 one can argue as in the proof of
 Theorem \ref{thm:2n-1dpi2n-2nonzero}.
 Under the assumption of aperiodicity
 of the $\alpha$-Reeb flow on the universal cover $M'$
 one shows with \cite[Sections 3 and 7]{gnw16}
 that $S$ is null-bordant
 via a deformation of the evaluation map
 defined on a surgered moduli space of holomorphic discs.
\end{proof}

\begin{rem}
 With the same reasoning
 and \cite[Sections 5]{gnw16}
 it follows that in the situation
 of Theorem \ref{thm:2n-1dhnnonzero}
 $(M,\omega)$ admits a contractible
 closed characteristic provided that
 the class $\{S\}$ of the belt sphere $S$
 is non-trivial in $\pi_3M'$ if $n=3$ and $k=2$,
 resp., in $\pi_4M'$ if $n=4$ and $k=3$.
\end{rem}

\begin{rem}
 The result of Theorem \ref{thm:2n-1dhnnonzero}
 is based on moduli spaces of holomorphic discs
 with boundary on a
 family of Legendrian open books
 in $(M',\xi)$.
 The involved arguments
 for the holomorphic curves
 in symplectisations of $M'$
 of the preceding section
 work equally well for
 Legendrian open books with boundary.
 Similarly to Theorem \ref{thm:3dovertwist}
 and \cite{ah09,gnw16,mnw13,nr11}
 one shows:
 If the $C^3$-norm of $\alpha$ is finite
 and if $(M',\xi)$ admits a Legendrian
 open book with boundary,
 then $(M,\omega)$ admits a contractible
 closed characteristic.
 Such examples can be obtained with
 \cite[Proposition 2.6.1 or Proposition 2.6.2]{wz16}.
\end{rem}


\subsection{Magnetic energy surfaces\label{subsec:magensur}}

Let $(Q,h)$ be a closed $n$-dimensional
Riemannian manifold
and denote by $\tau\co T^*Q\ra Q$
the cotangent bundle of $Q$.
Using the Levi-Civita connection $D^Q$ of $h$
and the dual metric $h^{\flat}$
we split the tangent bundle $TT^*Q=\HH\oplus\VV$
of $T^*Q$ into the horizontal and vertical distribution.
A metric $m$ on $T^*Q$ is defined via
\[
m\big((v,a),(w,b)\big):=
h\big(T\tau(v),T\tau(w)\big)+
h^{\flat}(a,b)
\]
for all $(v,a),(w,b)$ in $\HH\oplus\VV$.
The Levi-Civita connection of $m$ is denoted by $D$.
This turns $\tau$ into a Riemannian submersion
with totally geodesic fibres as is readily verified in
Riemannian normal coordinates.
Indeed,
straight vector space lines
contained in a cotangent fibre $T^*_qQ$
minimize the length
in the class of curves
connecting two given sufficiently close points
on $T^*_qQ$.

On $T^*Q$ we consider the twisted symplectic form
$\omega_{\sigma}=\rmd\lambda+\tau^*\sigma$,
where $\lambda$ denotes the Liouville $1$-form 
of $\tau$ and $\sigma$ is a magnetic $2$-form on $Q$,
and the Hamiltonian function
\[
H=\frac12|\,.\,|^2_{h^{\flat}}+V\circ\tau\,,
\]
where $V\co Q\ra\R$ is a potential function on $Q$.
For energies $e>\max_QV$ we denote the
regular level set $\{H=e\}\subset T^*Q$
by $M$, which is equipped with the odd-symplectic form
$\omega=\omega_{\sigma}|_{TM}$ and the metric
$g=m|_{TM}$.
The Levi-Civita connection on $(M,g)$
is denoted by $\nabla$
and the second fundamental form of
$(M,g)$ in $(T^*Q,m)$ by $\II$.

We assume that there exists a $1$-form
$\theta$ on the universal cover $\widetilde{Q}$
such that the lift $\mu^*\sigma=\rmd\theta$
of the magnetic form
along the universal covering map
$\mu\co\widetilde{Q}\ra Q$
has primitive $\theta$.
To $\widetilde{Q}$ we lift the metric $h$
via pull back $\tilde{h}=\mu^*h$
so that $\mu$ is turned into a local isometry.
The induced Levi-Civita
connection is denoted by
$\widetilde{D}^Q$.
In the same way the Hamiltonian $H$
lifts to $\widetilde{H}$ on $T^*\widetilde{Q}$
defining $M'$ via $\{\widetilde{H}=e\}$.
Along the induced universal covering map
$T^*\mu\co T^*\widetilde{Q}\ra T^*Q$
we lift the geometry of $(T^*Q,m,D)$
to $(T^*\widetilde{Q},\tilde{m},\widetilde{D})$
turning $T^*\mu$ into a local isometry.
This induces a metric $g'=\tilde{m}|_{TM'}$ on $M'$,
which is the lifted metric,
such that $\pi=T^*\mu|_{M'}$ is a local isometry.
The Levi-Civita connection of $g'$ is denoted by
$\nabla'$ and the second fundamental form of
$(M',g')$ in $(T^*\widetilde{Q},\tilde{m})$
is denoted by $\II'$.

The restriction
\[
\alpha=
\big(\tilde{\lambda}+\tilde{\tau}^*\theta\big)|_{TM'}\;,
\]
where $\tilde{\lambda}$ denotes the Liouville $1$-form 
of the cotangent bundle
$\tilde{\tau}\co T^*\widetilde{Q}\ra\widetilde{Q}$,
serves as a primitive of $\omega'=\pi^*\omega$.
By \cite[Proposition 2.4.1]{wz16}
the tuple
\[
\big(\pi\co M'\ra M,\alpha,\omega,g\big)
\]
is a virtually contact structure for all energies
$e>\sup_{\widetilde{Q}}\widetilde{H}(\theta)$
provided that $\|\theta\|_{C^0}$ is finite.

\begin{prop}
 \label{ckthetatockalpha}
 Let $k\in\N$ be a natural number.
 If the $C^k$-norm of $\theta$
 with respect to $(\widetilde{Q},\tilde{h},\widetilde{D}^Q)$
 is finite,
 then the $C^k$-norm of $\alpha$
 with respect to $(M',g',\nabla')$
 is finite too.
\end{prop}

\begin{proof}
 We abbreviate
 \[
 \beta=
 \tilde{\lambda}+\tilde{\tau}^*\theta\;.
 \]
 Because of the Gau{\ss} formula
 (cf.\ \cite[Theorem 1.72]{bes87})
 we obtain for vector fields
 $X,Y$ on $M'$ suitably extended
 to $T^*\widetilde{Q}$ that
 \[
 \big(\nabla'_X\alpha\big)(Y)=
 \big(\widetilde{D}_X\beta\big)(Y)|_{M'}+
 \beta\big(\II'(X,Y)\big)|_{M'}
 \;.
 \]
 For the extension
 we used a small tubular neighbourhood
 $\big\{e-\varepsilon<\widetilde{H}<e+\varepsilon\big\}$
 of $M'$ and the vector fields are required
 to be tangent to each level set of $\widetilde{H}$.
 Moreover,
 $\II'$ stands for the $(1,2)$-tensor
 that is defined as in \cite[(9.17)]{bes87}
 and coincides with the second fundamental form
 restricted to each energy surface.
 Similarly,
 \[
 \big(\nabla'_X(\nabla'\alpha)\big)(Y_1,Y_2)
 \]
 is in view of Section \ref{subsec:hordcovderiv}
 the sum of 
 \[
 \Big(
 \widetilde{D}_X
 \big(\widetilde{D}\beta\big)\Big)
 (Y_1,Y_2)|_{M'}
 \]
 and
 \[
 \big(\widetilde{D}\beta\big)
 \Big(\II'(X,Y_1),Y_2\Big)|_{M'}+
 \big(\widetilde{D}\beta\big)
 \Big(Y_1,\II'(X,Y_2)\Big)|_{M'}
 \]
 and
 \[
 \beta\circ\II'
 \Big(\II'(X,Y_1),Y_2\Big)|_{M'}+
 \beta\circ\II'
 \Big(Y_1,\II'(X,Y_2)\Big)|_{M'}
 \]
 as well as
 \[
 \Big(
 \widetilde{D}_X
 \big(\beta\circ\II'\big)
 \Big)
 (Y_1,Y_2)|_{M'}
 =
 \big(\widetilde{D}_X\beta\big)
 \big(\II'(Y_1,Y_2)\big)|_{M'}+
 \beta\Big(
 \big(\widetilde{D}_X\II'\big)(Y_1,Y_2)
 \Big)|_{M'}
 \;.
 \]
 Inductively one finds that
 $\big((\nabla')^k\alpha\big)(X,Y_1,\ldots,Y_k)$
 admits an analogous expression that involves
 $\widetilde{D}$-covariant derivatives of $\beta$
 up to order $k$ and $\widetilde{D}$-covariant
 derivatives of $\II'$ up to order $k-1$,
 cf.\ Appendix \ref{thirdcovderiv} for the third derivative.
 Invoking the local isometry $T^*\mu$,
 the $(1,\ell+2)$-tensor $\widetilde{D}^{\ell}\!\II'$
 along $M'$ can be estimated along the compact
 manifold $M$ by $D^{\ell}\!\II$,
 which is uniformly bounded,
 for all $\ell=0,1,\ldots,k-1$.
 
 Therefore,
 it is enough to estimate
 $\widetilde{D}^{\ell}\beta$
 on a disc subbundle
 $DT^*\widetilde{Q}$ of $T^*\widetilde{Q}$
 that contains $M'$.
 Observe that the local isometry
 $T^*\mu$ from
 $(T^*\widetilde{Q},\tilde{m},\widetilde{D})$
 to $(T^*Q,m,D)$
 pulls the Liouville form $\lambda$ of $\tau$
 back to the Liouville form
 $\tilde{\lambda}=(T^*\mu)^*\lambda$
 of $\tilde{\tau}$.
 Hence,
 using Remark \ref{rem:behaviorunderisometries}
 we get
 $\widetilde{D}^{\ell}\tilde{\lambda}=(T^*\mu)^*D^{\ell}\lambda$
 for all $\ell$.
 This implies that the $C^k$-norm of $\tilde{\lambda}$
 restricted to $DT^*\widetilde{Q}$
 is equal to the $C^k$-norm of $\lambda$
 restricted to $DT^*Q$
 relative to the respective geometries.
 As the latter is finite
 by the compactness of the disc bundle
 it remains to be shown finiteness
 of $\|\tilde{\tau}^*\theta\|_{C^k}$
 on $(DT^*\widetilde{Q},\tilde{m},\widetilde{D})$.
 
 We will estimate
 \[
 \gamma=\tilde{\tau}^*\theta
 \]
 assuming $C^k$-bounds on $\theta$.
 According to the splitting into
 horizontal and vertical spaces
 the first covariant derivative
 $\widetilde{D}\gamma$
 on tangent vectors
 $(v,a),(w,b)$ in $\HH\oplus\VV$
 at $u\in T^*\widetilde{Q}$
 equals to
 \[
 \big(\widetilde{D}\gamma\big)(v,w)+
 \big(\widetilde{D}\gamma\big)(v,b)+
 \big(\widetilde{D}\gamma\big)(a,w)+
 \big(\widetilde{D}\gamma\big)(a,b)
 \,.
 \]
 We will estimate each term of this expression separately.
 
 Denote by $X,Y$ horizontal vector fields
 that are lifts along $T\tilde{\tau}$
 of tangent vector fields $X^{\widetilde{Q}},Y^{\widetilde{Q}}$
 on $\widetilde{Q}$.
 We obtain
 \[
 \big(\widetilde{D}\gamma\big)(X,Y)|_u=
 \big(
 \widetilde{D}^Q\theta
 \big)
 (X^{\widetilde{Q}},Y^{\widetilde{Q}})|_{\tilde{\tau}(u)}
 \]
 because
 $\gamma(Y)=\theta(Y^{\widetilde{Q}})$
 is constant along the cotangent fibres
 and the horizontal part of $\widetilde{D}_XY$
 is projected to
 $\widetilde{D}^Q_{X^{\widetilde{Q}}}Y^{\widetilde{Q}}$
 via $T\tilde{\tau}$,
 see \cite[p.~240]{bes87}.
 Using the assumption of $\theta$ being $C^1$-bounded
 we infer that
 $\widetilde{D}\gamma$
 is finite on unit horizontal vectors $v,w$.
 
 For a horizontal vector field $X$
 and a vertical vector field $V$ we get
 \[
 \big(\widetilde{D}\gamma\big)(X,V)=
 -\theta
 \Big(
 T\tilde{\tau}\big(A_XV\big)
 \Big)
 \]
 because $T\tilde{\tau}(V)$ vanishes identically.
 Here $A$ is the $(1,2)$-tensor field on $T^*\widetilde{Q}$
 defined by decomposing $\widetilde{D}_XV$
 orthogonally into the respective horizontal and vertical parts
 as in \cite[Definition 9.20]{bes87} measuring
 the non-integrability of the horizontal distribution.
 Because $\theta$ is a bounded $1$-form on $\widetilde{Q}$,
 $T\tilde{\tau}$ is an orthogonal projection,
 and $A$ is a geometric tensor of the lifted geometry
 of $(T^*\widetilde{Q},\tilde{m},\widetilde{D})$,
 the tensor field
 $\widetilde{D}\gamma$ restricted to
 the chosen disc bundle $DT^*\widetilde{Q}$
 is bounded on mixed unit tangent vectors
 $v\in\HH_u$ and $b\in\VV_u$.
 
 Similarly,
 for a vertical vector field $U$
 and a lifted horizontal vector field $Y$
 \[
 \big(\widetilde{D}\gamma\big)(U,Y)=
 -\theta
 \Big(
 T\tilde{\tau}\big(A_YU\big)
 \Big)
 \]
 because
 $\gamma(Y_u)=\theta\big(Y^{\widetilde{Q}}_{\tilde{\tau}(u)}\big)$
 is constant along the cotangent fibres,
 so that an application of the vertical vector field $U$
 vanishes,
 and by \cite[Definition 9.21b]{bes87}
 combined with the fact that $[Y,U]$ is vertical,
 see \cite[p.~240]{bes87}.
 Again,
 finiteness of the $C^0$-norm of $\theta$
 implies that
 $\widetilde{D}\gamma$ is bounded
 on mixed unit tangent vectors
 $a\in\VV_u$ and $w\in\HH_u$
 for all $u\in DT^*\widetilde{Q}$.
 
 Finally,
 because $\gamma(V)$ vanishes identically
 and because the fibres of $\tilde{\tau}$
 are totally geodesic
 we obtain with \cite[9.25a and 9.26]{bes87}
 that
 \[
 \big(\widetilde{D}\gamma\big)(U,V)=0
 \]
 for vertical vector fields $U,V$.
 Summarizing,
 we obtain that $\|\gamma\|_{C^1}$
 is bounded by
 $2\|\theta\|_{C^1}\big(\|A\|_{C^0}+1\big)$
 on $DT^*\widetilde{Q}$.
 Moreover, we get that
 \[
 \big(\widetilde{D}\gamma\big)(X\oplus U,Y\oplus V)
 \]
 decomposes as
 \[
 \big(
 \widetilde{D}^Q\theta
 \big)
 \Big(T\tilde{\tau}(X),T\tilde{\tau}(Y)\Big)
 -\theta
 \Big(
 T\tilde{\tau}\big(A_XV\big)
 \Big)
 -\theta
 \Big(
 T\tilde{\tau}\big(A_YU\big)
 \Big)
 \]
 for all lifted horizontal vector fields $X,Y$
 and vertical vector fields $U,V$
 on $T^*\widetilde{Q}$
 with respect to $\HH\oplus\VV$.
 In other words,
 the tensor $\widetilde{D}\gamma$
 is equal to
 \[
 \widetilde{D}\gamma=
 \widetilde{D}^Q\theta\circ T\tilde{\tau}-
 \theta\circ T\tilde{\tau}\circ 2(A)
 \]
 by \cite[9.21]{bes87}
 denoting the symmetrization
 of the tensor $A$ by $(A)$.
 
 The second covariant derivative
 $\widetilde{D}^2\gamma$
 can be computed in a similar fashion.
 For any tangent vector field $X$ on
 $T^*\widetilde{Q}$
 we consider
 $\widetilde{D}_X\widetilde{D}\gamma$
 read as $\widetilde{D}_X$
 applied to the difference expression
 of $\widetilde{D}\gamma$.
 The minuend
 \[
 \widetilde{D}_X
 \Big(
 \widetilde{D}^Q\theta\circ T\tilde{\tau}
 \Big)
 \]
 can be treated as the first covariant derivative
 of $\gamma$ analogously
 to the above considerations.
 The subtrahend can be computed
 according to the product rule for
 the composition $\circ$ of tensors as
 \[
 \widetilde{D}_X
 \Big(
 \theta\circ T\tilde{\tau}\circ 2(A)
 \Big)=
 \widetilde{D}_X\big(\theta\circ T\tilde{\tau}\big)\circ 2(A)+
 \big(\theta\circ T\tilde{\tau}\big)\circ\widetilde{D}_X\big(2(A)\big)
 \,.
 \]
 Because of $\gamma=\tilde{\tau}^*\theta$
 this sum can be discussed
 with the arguments for $\gamma$ as above
 with the additional phenomenon that a
 covariant derivative of $A$,
 which still is a geometric tensor,
 has to be taken into account.
 Inductively,
 one shows finiteness
 of $\|\tilde{\tau}^*\theta\|_{C^k}$
 on $(DT^*\widetilde{Q},\tilde{m},\widetilde{D})$
 as claimed.
\end{proof}


\subsection{Truncating the magnetic field\label{subsec:truncthemagfiel}}

We will continue the considerations
and the use of the notation from
Section \ref{subsec:magensur}.
In the following we will recall the
truncation construction from
\cite[Section 2.5]{wz16}:
Let $U$ be an embedded closed $n$-disc in $Q$.
The image of the origin is denoted by $q\in Q$.
We choose $U$ so small
such that the decomposition of
$\mu^{-1}(U)$ into the connected components
$U^p$, $p\in\mu^{-1}(q)$, results into
a family of diffeomorphisms
$\mu|_{U^p}$ onto $U$,
which are in fact isometries.
Let $\chi$ be a cut off function on $Q$
that is identically $1$ on $Q\setminus U$
and vanishes on a disc neighbourhood $W$ of $q$.
We define a closed $2$-form $\hat{\sigma}$
on $Q$ that equals $\sigma$ on $Q\setminus U$
and $\rmd(\chi\theta_U)$ on $U$,
where $\theta_U$ is a primitive of $\sigma|_U$.
Notice,
that $\hat{\sigma}$ and $\sigma$
are cohomologous.
In fact,
$(\chi-1)\theta_U$ is a primitive
of the difference $\hat{\sigma}-\sigma$.
Therefore,
\[
\theta':=\theta+(\tilde{\chi}-1)\mu^*\theta_U
\]
is a primitive of $\mu^*\hat{\sigma}$,
where $\tilde{\chi}:=\mu^*\chi$.
The restriction $\theta'|_{W^p}$ of $\theta'$
to all connected components $W^p$
of $\mu^{-1}(W)$ is closed, and, hence, exact.
As in the proof of \cite[Proposition 2.5.1]{wz16}
we select primitive functions $f^p$ of
$\theta'|_{W^p}=\rmd f^p$ via the Poincar\'e Lemma.
Define a primitive $\hat{\theta}$ of $\hat{\sigma}$
to be the $1$-form on $\widetilde{Q}$
that coincides with $\theta'$ on
$\widetilde{Q}\setminus\mu^{-1}(W)$
and with $\rmd(\tilde{\chi}_Wf^p)$ on $W^p$
for all $p\in\mu^{-1}(q)$,
where $\tilde{\chi}_W$ is the $\mu$-pull back
of a cut off function $\chi_W$
that is equal to $1$ on $Q\setminus W$
and vanishes on a disc like neighbourhood
of $q$ contained in the interior of $W$.
We call the pair $(\hat{\theta},\hat{\sigma})$
a {\bf truncation} of $(\theta,\sigma)$.

\begin{lem}
 \label{lem:truncationckbounds}
 Let $k\in\N$ be a natural number.
 If the $C^k$-norm of $\theta$
 with respect to $(\widetilde{Q},\tilde{h},\widetilde{D}^Q)$
 is finite,
 then the $C^k$-norm of the truncation $\hat{\theta}$
 with respect to $(\widetilde{Q},\tilde{h},\widetilde{D}^Q)$
 is finite too.
\end{lem}

\begin{proof}
 First of all observe that $\theta'$
 is bounded in the $C^k$-norm
 as it is obtained from $\theta$ by adding
 the pull back of the compactly supported
 $1$-form $(\chi-1)\theta_U$ on $Q$
 along the local isometry $\mu$.
 That $\hat{\theta}$ is bounded in $C^0$
 was verified in \cite[Proposition 2.5.1]{wz16}.
 In order to estimate the $\ell$-th
 covariant derivative of $\hat{\theta}$ for $\ell\leq k$
 it suffices to do so for the restrictions
 $\rmd(\tilde{\chi}_Wf^p)$ of $\hat{\theta}$
 to $W^p$ for all $p\in\mu^{-1}(W)$.
 In view of Remark \ref{rem:ckforfunctions}
 we point out that the $(\ell+1)$-th
 covariant derivative of the products
 $\tilde{\chi}_Wf^p$ equals the sum of
 tensor products
 \[
 \sum_{j=0}^{\ell+1}
 \textstyle{{\ell+1\choose j}}\,
 \big(\widetilde{D}^Q\big)^j\tilde{\chi}_W
 \otimes
 \big(\widetilde{D}^Q\big)^{\ell+1-j}f^p
 \]
 according to the Leibniz rule.
 All covariant derivatives of $\tilde{\chi}_W$
 are bounded in view of Remark
 \ref{rem:behaviorunderisometries}.
 Moreover,
 a $C^0$-bound for $f^p$
 can be obtained as in
 \cite[Proposition 2.5.1]{wz16}.
 For $\ell+1-j\geq1$ we have by
 Remark \ref{rem:ckforfunctions}
 \[
 \big(\widetilde{D}^Q\big)^{\ell+1-j}f^p
 = 
 \big(\widetilde{D}^Q\big)^{\ell-j}\theta'|_{W^p}
 \,,
 \]
 which is bounded by the $C^k$-norm of $\theta'$
 and, hence, by the $C^k$-norm of $\theta$.
\end{proof}

\begin{rem}
\label{rem:remonsomecontconnsum}
If $\|\theta\|_{C^k}$ is finite,
then the primitive
\[
\hat{\alpha}=
\big(\tilde{\lambda}+\tilde{\tau}^*\hat{\theta}\big)|_{TM'}
\]
of $\pi^*{\hat{\omega}}$
is $C^k$-bounded,
where $\hat{\omega}:=\omega_{\hat{\sigma}}|_{TM}$,
and defines a virtually contact structure
\[
\big(\pi\co M'\ra M,\hat{\alpha},\hat{\omega},g\big)
\]
for all energies
$e>\sup_{\widetilde{Q}}\widetilde{H}(\hat{\theta})$
by \cite[Proposition 2.4.1]{wz16}
that is somewhere contact in the sense of
\cite[Section 2.5]{wz16}.
With the help of the covering connected sum
described in \cite[Section 2.2]{wz16}
a connected sum of somewhere contact
virtually contact structures
is defined.
Given two $C^k$-bounded
somewhere contact
virtually contact structures
the resulting virtually contact structure
on the connected sum
will be $C^k$-bounded too.
\end{rem}


\subsection{Classical Hamiltonians and magnetic fields\label{subsec:clahammagfie}}

In view of Section \ref{subsec:magensur}
let $(Q,h)$ be a $n$-dimensional
closed Riemannian manifold
that splits off a product of closed hyperbolic surfaces.
An $\R$-linear combination
of the corresponding area forms
coming from the surface factors
defines a magnetic form $\sigma$ on $Q$.
The lift of $\sigma$
to the universal cover $\widetilde{Q}$ of $Q$
has a primitive $1$-form $\theta$
that is the corresponding linear combination
of $\frac1y\rmd x$ for $(x,y)\in H^+$,
which is $C^k$-bounded
for all $k\in\N$ by the computations
in Section \ref{subsec:ucibanexample}.
If we replace $\sigma$
by a $2$-form on $Q$ whose
cohomology class is contained
in the subspace of $H_{\dR}^2Q$
spanned by the classes of the area forms
coming from the surface factors,
we obtain a magnetic form whose lift
admits a $C^{\infty}$-bounded primitive $\theta$.
This is because adding an exact $2$-form
on $Q$ keeps the property of having a bounded primitive
on $\widetilde{Q}$.
As discussed in Section \ref{subsec:magensur}
this leads to a rich class of examples of
virtually contact type
energy surfaces in classical mechanics with magnetic fields.

We would like to describe a particular class
of Hamiltonian systems on $(Q,h)$.
Let $\sigma$ and $\theta$ as above
and consider a Morse function $V$
on the manifold $Q$.
By composing $V$ with a strictly increasing function
we can arrange that:
\begin{enumerate}
\item
$V$ has a unique local maximum
which is assumed to be positive.
\item
All critical values of $V$
that correspond to critical points
of index $\leq n-1$
are strictly smaller than
\[
-\frac12 t_0^2\,,
\]
where $t_0:=\max_{\widetilde{Q}}|\theta|_{(\tilde{h})^{\flat}}$.
\item
\label{item3}
Denote by $c_{n-1}$ the largest
critical value  of $V$
not equal to the maximum of $V$.
Let $-v_0<0$ be a regular value of $V$
such that
\[
-v_0\in\left(c_{n-1},-\frac12 t_0^2\right)\,.
\]
We require that $\sigma$
vanishes on the $n$-disc
\[
\big\{V\geq-v_0\big\}
\subset Q
\]
and that $\theta$ vanishes on
\[
\big\{\widetilde{V}\geq-v_0\big\}
\subset\widetilde{Q}
\,,
\]
denoting by $\widetilde{V}$ the lift
of $V$ to $\widetilde{Q}$.
\end{enumerate}
Notice that item \eqref{item3} can be achieved
with truncation from
Section \ref{subsec:truncthemagfiel}.
It follows from \cite[Section 3.3]{wz16}
that $M=\{H=0\}$ is of virtually contact type
and somewhere contact.
We remark that the Ma\~n\'e critical value
of the magnetic system,
which is greater or equal than the maximum of $V$,
is positive.

\begin{proof}[{\bf Proof of Theorem \ref{thmintr:truncclasshammotion}}]
 We can assume the situation of the proceeding
 section for $n=2$.
 By \cite[Theorem 1.2]{wz16} the $3$-dimensional
 energy surface $M$ has non-vanishing $\pi_2M$.
 The existence of a periodic solution
 of the equations of motion that is contractible in $\{V\leq0\}$
 follows from Theorem \ref{thm:3dpi2nonzero}
 by projecting the obtained contractible periodic solution
 of $X_H$ on $M$ to $Q$ via $\tau$.
 The solution is non-constant according to the
 equations of motion.
 Indeed, a constant zero energy solution
 would be contained in the regular level set $\{V=0\}$ in $Q$,
 on which the magnetic form $\sigma$ vanishes,
 cf.\ \cite[p.~135]{sz16}.
\end{proof}

\begin{exwith}
\label{ex:reginflmagnontriv}
 We will describe magnetic Hamiltonian systems
 to which Theorem \ref{thmintr:truncclasshammotion}
 applies such that the obtained solution inside $\{V\leq0\}$
 cannot stay entirely in the annulus 
 $\big\{-v_0\leq V\leq0\big\}$,
 on which the magnetic form vanishes:
 Integrating gradient flow lines as in \cite[p.~153]{hir94}
 we find a diffeomorphism of $[-v_0,0]\times S^1$
 onto $\big\{-v_0\leq V\leq0\big\}$ such that
 the metric tensor $h_{ij}$ is given by a diagonal matrix
 with respect to the $(r,\theta)$-coordinates
 and such that $V(r,\theta)=r$.
 In fact,
 we can conformally change the metric $h$
 by multiplying with a function
 that restricted to the annulus equals
 the norm square of the gradient of $V$
 so that $h_{11}=1$.
 On $Q$ we choose magnetic fields subject to the requirements
 of Theorem \ref{thmintr:truncclasshammotion}.
 
 For the systems described we find
 a non-constant periodic solution $\gamma$
 of the equations of motion with $H(\gamma)=0$.
 Moreover,
 $\gamma$ is contractible in $\{V\leq0\}$.
 We claim that the trace of $\gamma$
 cannot be contained in the annulus
 $\big\{-v_0\leq V\leq0\big\}$.
 Otherwise we will reach a contradiction as follows:
 Via the gradient flow the solution $\gamma$
 can be homotoped into the boundary of $\{V\leq0\}$,
 which is a circle.
 The induced mapping degree of the homotoped
 $\gamma$ must vanish as otherwise a multiple
 of the boundary circle $\{V=0\}$ would be contractible in $\{V\leq0\}$.
 This implies that $\gamma$ is contractible
 inside the annulus.
 Therefore,
 we can consider the lift $\big(r(t),\theta(t)\big)$ of $\gamma$
 to the universal cover $[-v_0,0]\times\R$
 that is equipped with $(r,\theta)$-coordinates
 for which the metric tensor of $h$
 is diagonal such that $h_{11}=1$
 and $\grad V$ equals $\partial_r$.
 As $\theta$ has to have extremal points
 we find $t_0$ such that
 $\big(\dot{r}(t_0),\dot{\theta}(t_0)\big)=\big(\sqrt{-2r(t_0)},0\big)$.
 On the other hand,
 the Christoffel symbols
  $\Gamma_{11}^1$ and $\Gamma_{11}^2$
 of the Levi--Civita connection $D^Q$ of $h$ vanish
 such that the curve
 $\beta(t)=\big(b(t),\theta(t_0)\big)$
 for a quadratic polynomial $b(t)$ with
 leading term $-\frac12t^2$,
 and with $b(t_0)=r(t_0)$ and
 $\dot{b}(t_0)=\sqrt{-2r(t_0)}$
 is a solution of the equation of motion
 $D^Q_{\dot{\beta}}\dot{\beta}=-(1,0)$
 of zero energy.
 Using uniqueness of solutions of
 second order ordinary differential equations
 for given initial data
 and an open--closed argument
 we see that the periodic solution
 $\gamma\co\R\ra[-v_0,0]\times\R$
 coincides with $\beta$ where it is defined.
 In particular, the trace of $\gamma$
 connects the boundary components of the annulus
 along a gradient flow line of $V$
 so that $\gamma$ has vanishing velocity
 at the intersection with the lower boundary
 corresponding to $-v_0$.
 This is a contradiction.
\end{exwith}

\begin{proof}[{\bf Proof of Theorem \ref{thmintr:hightruncclasshammotion}}]
 We assume that $n\geq3$.
 The proof is based on the fact
 that the energy surface
 $M=\{H=0\}$ is of virtually contact type
 such that the upper boundary
 of the standard $(n-1)$-handle $D^{n-1}\times D^{n+1}$
 admits a contact embedding into $(M',\ker\alpha)$
 such that the image of the belt sphere
 represents a non-trivial homology class
 of degree $n$ in $M'$.
 
 We give a direct verification
 of the contact type property of $M'$.
 For this denote by $X$
 the gradient vector field of $\tau^*V$
 with respect to the metric $m$ on $T^*Q$
 and define a function $F$ on $T^*Q$
 by setting $F=\lambda(X)$.
 We remark that $F(u)=h^{\flat}(u,\rmd V)$
 for all covectors $u$ on $Q$.
 Lifting to the universal cover $T^*\widetilde{Q}$
 we obtain for all $\varepsilon>0$
 and covectors $\tilde{u}$ on $\widetilde{Q}$ that
 \[
 \Big(\tilde{\lambda}+\tilde{\tau}^*\theta-\varepsilon\rmd\widetilde{F}\Big)
 \big(X_{\widetilde{H}}\big)(\tilde{u})
 \]
 is equal to the sum of
 \[
 |\tilde{u}|^2_{(\tilde{h})^{\flat}}+
 (\tilde{h})^{\flat}(\tilde{u},\theta)
 \]
 and $\varepsilon$ times
 \[
 -\big(\!\Hess_{\tilde{h}}\widetilde{V}\big)
 (\tilde{u}^{\#},\tilde{u}^{\#})
 +\big|\!\grad_{\tilde{h}}\widetilde{V}\big|^2_{\tilde{h}}
 +\big(\tau^*\mu^*\sigma\big)\Big(\tilde{u}^{\#},\grad_{\tilde{h}}\widetilde{V}\Big)
 \,,
 \]
 where $X_{\widetilde{H}}$ denotes the Hamiltonian vector field
 of the system $(\tilde{\omega}_{\rmd\theta},\widetilde{H})$,
 $\tilde{u}^{\#}$ is the $\tilde{h}$-dual vector of $\tilde{u}$,
 and $\Hess_{\tilde{h}}\widetilde{V}=\widetilde{D}^Q\rmd\widetilde{V}$
 is the Hessian of $\widetilde{V}$.
 Under the assumption of the theorem
 we can assume the situation
 described at the beginning of this section.
 Distinguishing the cases
 $\frac12|\tilde{u}|^2_{(\tilde{h})^{\flat}}\geq v_0$
 and
 $\frac12|\tilde{u}|^2_{(\tilde{h})^{\flat}}<v_0$
 one shows that the above sum
 is uniformly positive along
 \[
 \left\{
 \frac12|\tilde{u}|^2_{(\tilde{h})^{\flat}}=
 -\widetilde{V}\big(\tilde{\tau}(\tilde{u})\big)
 \right\}
 \]
 for some small $\varepsilon>0$.
 It follows that $M'$ is of virtually contact type.
 
 In fact, taking the family $t\theta$ of $1$-forms
 for $t\in[0,1]$,
 which corresponds to the family $t\sigma$
 of magnetic forms, it follows that
 \[
 \alpha_t=
 \big(
 \tilde{\lambda}+t\cdot\tilde{\tau}^*\theta-\varepsilon\rmd\widetilde{F}
 \big)|_{TM'}
 \]
 is a family of contact forms on $M'$
 that connects $\alpha=\alpha_1$
 with a contact form $\alpha_0$,
 which descends to the contact form
 $\alpha_W:=(\lambda-\varepsilon\rmd F)|_{TM}$
 on $M$.
 The induced contact structures
 are contactomorphic along a global flow
 as the Gray stability
 argument in \cite[p.~60/61]{gei08} shows.
 Indeed, the Moser integrating time-dependent
 vector field on $M'$ is bounded
 with respect to the complete metric $g'$
 because $\alpha_t$ is bounded in $C^1$
 (see Sections \ref{subsec:magensur}
 and \ref{subsec:truncthemagfiel})
 as local arguments used
 in Section \ref{subsec:indconvcpxstr} show.
 By \cite[Example 11.12 (2)]{ce12}
 $(M,\ker\alpha_W)$ is the contact type
 boundary of the Weinstein handle body
 $\{H\leq0\}$ in $(T^*Q,\rmd\lambda)$
 with Morse function $H$ and Liouville vector field
 $\mathbf{p}\partial_{\mathbf{p}}+\varepsilon X_F$
 because
 \[
 \rmd H\big(\mathbf{p}\partial_{\mathbf{p}}+\varepsilon X_F\big)
 =(\lambda-\varepsilon\rmd F)(X_H)
 \]
 on $(T^*Q,\rmd\lambda)$.
 By \cite[Section 3.1]{wz16} and
 the assumptions on the potential $V$
 the Weinstein handle body $W$ is subcritical.
 
 Up to Weinstein handle moves
 according to the Morse--Smale theory
 for Weinstein structures in
 \cite[Chapters 10, 12, and 14]{ce12}
 we can assume that the largest critical value
 is taken at precisely one critical point $p_0$
 and the Morse index of $p_0$ equals $n-1$.
 In view of Gray stability
 this changes the contact structure
 on the boundary by a contact isotopy.
 With \cite[Proposition 12.12]{ce12}
 we can assume after a Weinstein isotopy
 supported in a neighbourhood of $p_0$
 that the Weinstein structure near $p_0$
 is given by the standard model handle
 as it appears in contact surgery,
 cf.\ \cite{gei08}.
 Replacing $\partial W$ by a regular level set
 slightly above the largest critical value,
 what results in a further
 contact isotopy of $(M,\ker\alpha_W)$
 by following the Liouville flow,
 the upper boundary of the model handle
 can be assumed to be contained
 in the boundary of the Weinstein handle body.
 After all we obtain a contact embedding
 of the upper boundary $D^{n-1}\times S^n$
 of the standard handle
 into $(M,\ker\alpha_W)$ whose belt sphere
 represents a non-zero homology class by
 \cite[Section 3.2]{wz16}.
 Hence, $D^{n-1}\times S^n$
 lifts to a contact handle in $(M',\ker\alpha)$
 so that the image of $S^n$ is non-zero
 in the homology of $M'$.
 In view of Theorem \ref{thm:2n-1dhnnonzero}
 and \cite[p.~135]{sz16}
 this proves the theorem.
\end{proof}


\appendix
\section{The third covariant derivative\label{thirdcovderiv}}

In view of the proof of
Proposition \ref{ckthetatockalpha}
we remark that the third covariant derivative
$\nabla'_X\big((\nabla')^2\alpha\big)$
evaluated on the triple of vector fields
$(Y_1,Y_2,Y_3)$
is the sum of the following terms
after restricting to $M'$:
\[
\Big(
\widetilde{D}_X\big(\widetilde{D}^2\beta\big)\Big)
(Y_1,Y_2,Y_3)
\]
and $\widetilde{D}^2\beta$
evaluated on the triples
\[
\big(\II'(X,Y_1),Y_2,Y_3\big)
\quad\text{and}\quad
\big(Y_1,\II'(X,Y_2),Y_3\big)
\quad\text{and}\quad
\big(Y_1,Y_2,\II'(X,Y_3)\big)
\]
and $\widetilde{D}_X\big(\widetilde{D}\beta\big)$
evaluated on the tuples
\[
\big(\II'(Y_1,Y_2),Y_3\big)
\quad\text{and}\quad
\big(Y_1,\II'(Y_2,Y_3)\big)
\quad\text{and}\quad
\big(Y_2,\II'(Y_1,Y_3)\big)
\]
and $\widetilde{D}\beta$
evaluated on the tuples
\[
\big(\II'(Y_1,Y_2),\II'(X,Y_3)\big)
\;\text{,}\quad
\big(\II'(X,Y_2),\II'(Y_1,Y_3)\big)
\;\text{,}\quad
\big(\II'(X,Y_1),\II'(Y_2,Y_3)\big)
\]
and $\widetilde{D}\beta$
evaluated on the tuples
\[
\Big(\II'\big(X,\II'(Y_1,Y_2)\big),Y_3\Big)
\;\text{,}\quad
\Big(Y_2,\II'\big(X,\II'(Y_1,Y_3)\big)\Big)
\;\text{,}\quad
\Big(Y_1,\II'\big(X,\II'(Y_2,Y_3)\big)\Big)
\]
and $\beta\circ\II'$
evaluated on the tuples
\[
\Big(\II'(Y_1,Y_2),\II'(X,Y_3)\Big)
\quad\text{and}\quad
\Big(\II'(X,Y_2),\II'(Y_1,Y_3)\Big)
\]
and $\beta\circ\II'$
evaluated on the tuples
\[
\Big(\II'\big(X,\II'(Y_1,Y_2)\big),Y_3\Big)
\quad\text{and}\quad
\Big(Y_2,\II'\big(X,\II'(Y_1,Y_3)\big)\Big)
\]
and
\[
\big(\widetilde{D}_X\beta\big)
\Big(\big(\widetilde{D}_{Y_1}\II'\big)(Y_2,Y_3)\Big)+
\beta\Big(\II'\big(X,\big(\widetilde{D}_{Y_1}\II'\big)(Y_2,Y_3)\big)\Big)
\]
as well as
\[
\Big(\widetilde{D}_X\big(\beta\circ\II'\big)\Big)
\big(\II'(Y_1,Y_2),Y_3\big)
\,,
\]
which equals,
\[
\big(\widetilde{D}_X\beta\big)
\Big(\II'\big(\II'(Y_1,Y_2),Y_3\big)\Big)+
\beta \Big(\big(\widetilde{D}_X\II')\big(\II'(Y_1,Y_2),Y_3\big)\Big)
\]
and
\[
\Big(\widetilde{D}_X\big(\beta\circ\II'\big)\Big)
\big(Y_2,\II'(Y_1,Y_3)\big)
\,,
\]
which equals,
\[
\big(\widetilde{D}_X\beta\big)
\Big(\II'\big(Y_2,\II'(Y_1,Y_3)\big)\Big)+
\beta \Big(\big(\widetilde{D}_X\II')\big(Y_2,\II'(Y_1,Y_3)\big)\Big)
\,.
\]


\begin{ack}
  We would like to thank Alberto Abbondandolo,
  Peter Albers, Urs Frauenfelder,
  Hansj\"org Geiges, Otto van Koert, Yong-Geun Oh, Felix Schlenk,
  Jay Schneider, Stefan Suhr,
  Anna-Maria Vocke, Rui Wang.
\end{ack}


\end{document}